\RequirePackage[l2tabu, orthodox]{nag}
\documentclass[11pt]{amsart} %
\author[A. Hammerlindl]{Andy Hammerlindl}
\address{School of Mathematical Sciences, Monash University, Victoria 3800 Australia} \urladdr{ http://users.monash.edu.au/~ahammerl/}  \email{andy.hammerlindl@monash.edu}

\author[R. Potrie]{Rafael Potrie}
\address{CMAT, Facultad de Ciencias, Universidad de la Rep\'ublica, Uruguay}
\urladdr{www.cmat.edu.uy/$\sim$rpotrie}\email{rpotrie@cmat.edu.uy}

\author[M. Shannon]{Mario Shannon}
\address{CMAT, Facultad de Ciencias, Universidad de la Rep\'ublica, Uruguay \ AND \
Institute Math\`ematique de Burgogne, Dijon, France}
\email{shannon@cmat.edu.uy}

\title[Seifert manifolds admitting partially hyperbolic diffeomorphisms]{Seifert manifolds admitting partially hyperbolic diffeomorphisms}
\thanks{R.P.~and M.S.~were partially supported by CSIC group 618.
A.H.~and R.P.~were partially supported by the Australian Research Council.
R.P. was also partially supported by MathAmSud-Physeco.}

\date{\today}

\usepackage{amsmath}
\usepackage{amssymb}
\usepackage{amsfonts}
\usepackage{amsthm}
\usepackage{amscd}
\usepackage{graphicx}
\usepackage{setspace}
\usepackage{cleveref}
\usepackage{fourier}
\usepackage[T1]{fontenc}

\makeatletter
\def\saveenum{\xdef\@savedenum{\the\c@enumi\relax}}
\def\resetenum{\global\c@enumi\@savedenum}
\makeatother

\newcommand{\bbC}{\mathbb{C}}
\newcommand{\bbD}{\mathbb{D}}
\newcommand{\bbR}{\mathbb{R}}

\newcommand{\bbZ}{\mathbb{Z}}

\newcommand{\Es}{E^s}
\newcommand{\Ec}{E^c}
\newcommand{\Eu}{E^u}
\newcommand{\Ecu}{E^{cu}}
\newcommand{\Ecs}{E^{cs}}

\newcommand{\cW}{\mathcal{W}}
\newcommand{\Ws}{W^s}

\newcommand{\Wu}{W^u}

\newcommand{\Wup}{W^u_+}
\newcommand{\Wcp}{W^c_+}
\newcommand{\Wsp}{W^s_+}

\newcommand{\inv}{^{-1}}

\newcommand{\cF}{\mathcal{F}}
\newcommand{\Fs}{\mathcal{F}^s}

\newcommand{\Fu}{\mathcal{F}^u}
\newcommand{\Fcu}{\mathcal{F}^{cu}}
\newcommand{\Fcs}{\mathcal{F}^{cs}}

\newcommand{\tFs}{\tilde{\mathcal{F}^s}}

\newcommand{\tFu}{\tilde{\mathcal{F}^u}}

\newcommand{\sans}{\setminus}
\newcommand{\subof}{\subset}

\newcommand{\invn}{^{-n}}

\newcommand{\Lam}{\Lambda}

\newcommand{\gam}{\gamma}
\newcommand{\Sig}{\Sigma}

\newcommand{\eps}{\varepsilon}

\newcommand{\al}{\alpha}

\newcommand{\bt}{\beta}
\newcommand{\UTSig}{T^1 \Sig}

\newcommand{\dist}{\operatorname{dist}}

\numberwithin{equation}{section}
\newtheorem{thm}[equation]{Theorem}
\newtheorem{teo}[equation]{Theorem}
\newtheorem{cor}[equation]{Corollary}
\newtheorem{lemma}[equation]{Lemma}
\newtheorem*{teo*}{Theorem}
\newtheorem*{prop*}{Proposition}
\newtheorem*{cor*}{Corollary}
\newtheorem*{goal*}{Goal}

\newtheorem*{teoA}{Theorem A}
\newtheorem*{teoB}{Theorem B}

\newtheorem*{teoC}{Theorem C}

\newtheorem{quest}{Question}
\newtheorem*{af}{Claim}

\newtheorem{prop}[equation]{Proposition}

\theoremstyle{remark}
\newtheorem*{remark} {\textbf{Remark}}
\newtheorem{fact}[equation]{\textbf{Fact}}

\begin{document}

\maketitle

\begin{abstract}
We characterize which 3-dimensional Seifert manifolds admit transitive
partially hyperbolic diffeomorphisms.
In particular, a circle bundle over a higher-genus surface
admits a transitive partially hyperbolic diffeomorphism
if and only if it admits an Anosov flow.
\end{abstract}

\section{Introduction}\label{SectionIntroduccion}
This paper deals with the problem of classification of partially hyperbolic diffeomorphisms in dimension 3. This program was initiated by the works of Bonatti-Wilkinson (\cite{BonattiWilkinson}) and Brin-Burago-Ivanov (\cite{BBI}). These results were also motivated by an informal conjecture, due to Pujals, which suggested that there were no new transitive partially hyperbolic examples to discover in dimension 3 (see \cite{BonattiWilkinson,CHHU,HP-Survey}).  

The first topological obstructions to the existence of partially hyperbolic diffeomorphisms comes from \cite{BI} where the existence of partially hyperbolic diffeomorphisms in $S^3$ is excluded. Other topological obstructions related with the existence of partially hyperbolic diffeomorphisms in 3-manifolds with fundamental group with polynomial growth can be derived from the work of \cite{BI}; see also \cite{Parwani}. In \cite{HPNil,HPSol}, the first two authors provided a classification of partially hyperbolic diffeomorphisms in 3-manifolds with solvable fundamental group giving an affirmative answer to Pujals' conjecture in such manifolds. Notice that these include Seifert fiber spaces whose fundamental group is of polynomial growth, in particular, circle bundles over the sphere and the torus.

Even though Pujals' conjecture turned out to be false (\cite{BPP,BGP,BGHP}), there is still (at least) one pertinent question which remains open and is in the spirit of this conjecture.

\begin{quest}
If a 3-manifold whose fundamental group is of exponential growth admits a partially hyperbolic diffeomorphism, does it admit an Anosov flow? 
\end{quest}

Since the examples in \cite{BPP,BGP,BGHP} are obtained starting with an Anosov flow, they do not provide a negative answer to the previous question. This question seems hard in view that for the moment we do not know which 3-manifolds admit Anosov flows and more or less the same obstructions we know for the existence of Anosov flows are obstructions for the existence of partially hyperbolic dynamics. 

This paper grew out of the motivation to answer the question in a specific family of manifolds: circle bundles over higher genus surfaces, or more generally, Seifert fiber spaces. The simplest sub-family where the question of existence of transitive partially hyperbolic diffeomorphisms was unknown was $\Sigma_g \times S^1$ where $\Sigma_g$ is a closed surface of genus $g \geq 2$. By a classical result of Ghys (\cite{Ghys}) we know that such a manifold does not admit Anosov flows. %

We deal in this paper with partially hyperbolic diffeomorphisms in Seifert fiber spaces. For notational  simplicity, and for the convenience of the reader not familiar with Seifert fiberings in all generality, we start by studying the (orientable) circle bundle case which is more elementary, though most of the main ideas extend rather directly to the Seifert fiber setting. (Others require some further study and we finish the paper with a section which extends the results to the more general setting.) The strategy is to show that the dynamical foliations that a partially hyperbolic diffeomorphism carries are (after isotopy) transverse to the circle fibers. This allows us to provide some obstructions on the topology of the circle bundle that we expose later. Let us mention that this work also motivates the study of \emph{horizontal foliations} on Seifert fiber spaces to understand partially hyperbolic dynamics and this is very much related with the study of representations of surface groups in $\mathrm{Homeo}
(
S^1)$, a rather active subject whose interest has been renewed in the recent years. (See e.g. \cite{Bowden,Mann} and references therein.) 

The consequence of our main results which is easiest to state is: 

\begin{teo*}
If a circle bundle over a surface $\Sigma$ of genus $\geq 2$ admits a transitive partially hyperbolic diffeomorphism, then it admits an Anosov flow. In particular, $\Sigma \times S^1$ does not admit transitive partially hyperbolic diffeomorphisms. 
\end{teo*}  

The obstruction can be stated in terms of the \emph{Euler number} of the fiber bundle and says that a circle bundle over a higher genus surface admits a transitive partially hyperbolic diffeomorphism if and only if its Euler number divides the Euler characteristic of $\Sigma$. See Theorem A below. 

In the Anosov flow case, the corresponding statement for $\Sigma \times S^1$ follows from the classification result of Ghys \cite{Ghys} (see also \cite{Barbot} for the extension to Seifert fiber spaces). The proof we give here does not provide a classification, but is different even in the Anosov flow case. The results in this paper were partially announced in \cite{HP-Survey}. 

\subsection{Precise Statements of Results} 

\subsubsection{Circle bundles} Consider $M$ an orientable circle bundle over an orientable surface $\Sigma$ of genus $g\geq 2$.  See subsection \ref{ss.circlebundles} for precise definitions. 

Every circle bundle over $\Sigma$ can be obtained from $\Sigma \times S^1$ by removing a solid torus of the form $D \times S^1$ (with $D \subset \Sigma$ a two dimensional disk) and regluing by a map preserving the vertical fibers and giving rise to a new 3-manifold $M$ which is a circle bundle over $\Sigma$ in a non-trivial way. The number of turns a meridian of the form $\partial D \times \{t\}$ gives around a fiber after being sent by the gluing map is called the Euler number of the bundle and is denoted by $\mathrm{eu}(M)$ (see \cite[Chapter 4 of Book II]{CandelConlon} or \cite{Hatcher} for more detailed information about this concept). 

The unit tangent bundle $T^1 \Sigma$ over $\Sigma$ is well known to satisfy $\mathrm{eu}(T^1\Sigma) = \chi (\Sigma)$ where $\chi(\Sigma)$ denotes the Euler characteristic. Also, it is direct to show that $\mathrm{eu}(\Sigma \times S^1)=0$. 

The main result in this context is the following: 

\begin{teoA}
Let $M$ be an orientable circle bundle over $\Sigma$, an orientable surface of genus $g\geq 2$, admitting a transitive partially hyperbolic diffeomorphism. Then, there exists $k\in \mathbb{Z}$ such that $\mathrm{eu}(M) = \frac{\chi(\Sigma)}{k} \in \mathbb{Z}$. 
\end{teoA}

\begin{remark} This condition on the Euler number is equivalent to say that the circle bundle is a finite cover of the unit tangent bundle. 
\end{remark}

See subsection \ref{ss.PH} for the definition of partially hyperbolic diffeomorphism. In particular, we require that the tangent space $TM$ of $M$ splits into three non-trivial bundles, one uniformly contracting, one uniformly expanded and the third one satisfying a domination condition with respect to the other two. This is sometimes called in the literature \emph{pointwise strong partial hyperbolicity}. We say that a diffeomorphism $f$ is \emph{transitive} if there is a point $x\in M$ whose $f$-orbit is dense. 

A direct consequence is the following:

\begin{cor}Let $M$ be a circle bundle over $\Sigma$, a surface of genus $g\geq 2$, then $M$ admits a transitive partially hyperbolic diffeomorphism if and only if it admits an Anosov flow. 
\end{cor}

\begin{proof} The time one map of an Anosov flow is partially hyperbolic and if the flow is transitive and not a suspension, then the time one map is also transitive. The main result of \cite{Ghys} implies that an Anosov flow in a circle bundle must be transitive (and not a suspension). Thus, the reverse implication is trivial. The first one follows directly since by lifting the geodesic flow in negative curvature to finite covers one can construct Anosov flows in all bundles with the corresponding Euler numbers admitted by Theorem A. 
\end{proof}

There are many conditions weaker than transitivity that also allow us to obtain the same result. These appear in the statement of Theorem \ref{teo-novert} below, so beyond transitivity we also obtain: 

\begin{teoB}
Let $M$ be a circle bundle over $\Sigma$, a surface of genus $g\geq 2$, admitting a partially hyperbolic diffeomorphism which
\begin{itemize}
\item either is \emph{dynamically coherent} or 
\item \emph{homotopic to identity}. 
\end{itemize} 
Then, there is $k \in \mathbb{Z}$ such that $\mathrm{eu}(M) = \frac{\chi(\Sigma)}{k} \in \mathbb{Z}$. 
\end{teoB}

See subsection \ref{ss.PH} for a definition of dynamical coherence: it is related to the integrability properties of the bundles involved in the partially hyperbolic splitting. 

\subsubsection{Seifert fiber spaces}

We now consider the more general case of partially hyperbolic systems defined
on Seifert fiber spaces.
We give a brief description of Seifert fiberings and orbifolds here
and refer the reader to 
\cite{sco1983geometries} and \cite{cho2012geometric}
for details.

Consider the manifold with boundary $\bbD \times [0,1]$
equipped with a fibering where
each fiber is of the form $\{z\} \times [0,1]$ for $z \in \bbD$.
Here $\bbD$ is the unit disc in $\bbC$.
A \emph{standard fibered torus} is a solid torus
defined by gluing $\bbD \times \{1\}$ to $\bbD \times \{0\}$
by a map of the form
\[
    R_{p/q}(z,1) = (z \cdot e^{2 \pi i p/q}, 0)
\]
where $p/q$ is a rational number.
If $p/q$ is an integer, we call this a
\emph{trivial} fibered torus.
A \emph{Seifert fibering} on a closed 3-manifold $M$
is a smooth foliation of $M$ by circles
such that each point $x$ in $M$ has a neighbourhood
where the foliation is equivalent to that
of a standard fibered torus.
The fiber through $x$ is \emph{regular} if there is a
neighbourhood of $x$ where the foliation is equivalent to that of
a trivial fibered torus.
If a fiber is not regular, it is called \emph{exceptional}.
As $M$ is compact,
it has finitely many exceptional fibers.

Roughly speaking,
an orbifold is an object that locally resembles
a quotient of $\bbR^n$ by a finite group.
This paper only considers 2-dimensional orbifolds
associated to Seifert fiberings
and we may use the following definition specialized to this case.
First, for an integer $\al \ge 1,$
let $\bbD_\al$ be the quotient of $\bbD$
by the rotation $z \mapsto z \cdot e^{2 \pi i/\al}.$
Note that $\bbD_\al$ is a topological space
and $\bbD_\al \sans \{0\}$ is a smooth surface.
An \emph{orbifold} $\Sig$ is a closed topological surface
equipped with an atlas of charts such that
each chart is of the form
$(\varphi, U, \al)$
where $\al \ge 1$ is an integer,
$U \subof \Sig$ is open, and
$\varphi : U \to \bbD_\al$ is a homeomorphism;
the charts cover $\Sig;$ and
if $(\varphi, U, \al_U)$ and $(\psi, V, \al_V)$
are distinct charts in the atlas,
then $\varphi(U \cap V)$ and $\psi(U \cap V)$
are smooth manifolds and
\begin{math}
    \psi \circ \varphi \inv : \varphi(U \cap V) \to \psi(U \cap V)
\end{math}
is a diffeomorphism.
(In particular,
if $\al_U > 1$ then $\varphi \inv(0) \notin V$ and
if $\al_V > 1$ then $\psi \inv(0) \notin U.$)

We call $x \in \Sig$ a \emph{cone point} if the atlas has a chart
$(\varphi, U, \al)$ such that $\varphi(x) = 0$ and $\al > 1.$
If $M$ is a 3-manifold with a Seifert fibering,
let $\Sig$ be the topological space obtained by quotienting
each circle down to a point and note that $\Sig$
naturally has the structure of an orbifold.

Apart from a few known ``bad'' orbifolds,
every orbifold may be equipped with a uniform geometry which is either
elliptic, parabolic, or hyperbolic
\cite[Theorems 2.3 and 2.4]{sco1983geometries}.
Even though an orbifold $\Sig$
does not have the structure of a smooth manifold,
its unit tangent bundle
$\UTSig$ is a well-defined smooth 3-manifold and the natural projection
$\UTSig  \to  \Sig$ defines a Seifert fibering
\cite[\S 5]{ehn1981transverse}.
One may also define the geodesic flow on $\UTSig$.

%
%
If an orbifold is bad, elliptic, or parabolic, then
a Seifert fiber space over this orbifold
has one of the model geometries corresponding to
having virtually nilpotent fundamental group
\cite[Theorem 5.3]{sco1983geometries}.
Partially hyperbolic systems in these geometries are already
well understood
\cite{HPSol}.
Hence, we only consider hyperbolic orbifolds here.
Theorems {A} and {B} then have the following generalization.

\begin{teoC}%
    Let $M$ be a Seifert fiber space over a hyperbolic orbifold $\Sig$
    such that $M$ admits a partially hyperbolic diffeomorphism which is either
    transitive, dynamically coherent, or homotopic to the identity.
    Then $M$ finitely covers the unit tangent bundle of $\Sig$.
\end{teoC}
Barbot showed that such
manifolds are exactly the Seifert fiber spaces which support Anosov flows
\cite{Barbot}.

\begin{cor} \label{cor:seifanosov}
    A Seifert fiber space over a hyperbolic orbifold
    admits a transitive partially hyperbolic diffeomorphism
    if and only if it admits an Anosov flow.
\end{cor}

Some of the most studied
orbifolds are the so-called ``turnovers.''
An orbifold is a \emph{turnover}
if it has exactly three cone points and
the underlying topological space is a sphere.
Seifert fiberings over turnovers were the last family for which the
Milnor-Wood inequalities were generalized \cite{nai1994foliations},
and Brittenham showed specifically
that these manifolds do not support foliations with vertical leaves
\cite{Brittenham}.
In this setting, we may therefore state a result which does not rely on
transitivity or any dynamical assumption other than partial hyperbolicity.

\begin{thm} \label{thm:turnover}
    A Seifert fiber space over a turnover
    admits a partially hyperbolic diffeomorphism
    if and only if it admits an Anosov flow.
\end{thm}
\begin{cor} \label{cor:turnex}
    There are infinitely many Seifert fiber spaces which
    support horizontal foliations,
    but do not support partially hyperbolic diffeomorphisms.
\end{cor}

\subsection{Organization of the paper} Section \ref{s.prelim} introduces some known facts that we will use in the course of the proof of our main results. In section \ref{s.novert} we state the result which establishes, under some assumptions (including being transitive, or dynamically coherent) that the dynamical (branching) foliations carried by partially hyperbolic diffeomorphisms must be (homotopically) transverse to the fibers of the circle bundle and work out some preliminaries for the proof.  This result is proved in sections \ref{sec-coherent} and \ref{sec-novertproof}. Then, in section \ref{s.proofCircle} we use this fact to give the obstructions in the Euler number of the bundle to admit such partially hyperbolic diffeomorphisms. Finally, in section \ref{s.Seifert} we extend the results to the general Seifert fiber case. 

\medskip 
{\bf Acknowledgements:} We thank Christian Bonatti for several remarks and ideas that helped in this project and in particular for his help in the proof of Lemma \ref{lem-twopoints}. We also thank Harry Baik, Jonathan Bowden, Joaqu\'in Brum, Katie Mann and Juliana Xavier for several helpful discussions.

\section{Preliminaries}\label{s.prelim}
\subsection{General facts on partially hyperbolic dynamics}\label{ss.PH}

Let $M$ be a 3-dimensional manifold and $f: M \to M$ a $C^1$-diffeomorphism. We say that $f$ is \emph{partially hyperbolic} if there exists a $Df$-invariant continuous splitting $TM=\Es \oplus \Ec \oplus \Eu$ into $1$-dimensional subbundles and $N>0$ such that for every $x\in M$: 
\[
    \|D_xf^N|_{\Es} \| < \min\{ 1, \|D_xf^N|_{\Ec} \|\} \leq \max \{ 1, \|D_xf^N|_{\Ec} \|\} < \|D_x f^N|_{\Eu}\| .
\]
It is \emph{absolutely partially hyperbolic} if moreover, there exist constants $0<\sigma<1<\mu$ such that:
\[
    \|D_xf^N|_{\Es} \| < \sigma <  \|D_xf^N|_{\Ec}\|  < \mu < \|D_x f^N |_{\Eu}\| .
\]
It is possible to change the Riemannian metric so that $N=1$ and we will do so. See \cite{Gourmelon}. 

One calls $\Es$ and $\Eu$ the \emph{strong stable} and \emph{strong unstable} bundles respectively. These integrate uniquely and give rise to foliations $\Fs$ and $\Fu$ called respectively the \emph{strong stable} and \emph{strong unstable} foliations (\cite{HPS}). The leaf of $\Fs$ or $\Fu$ through a point $x$ will be denoted by $\Ws(x)$ or $\Wu(x)$. The \emph{center bundle} $\Ec$ may not be integrable into a foliation (though being one-dimensional it always admits integral curves). 

When the bundles $\Es \oplus \Ec$ and $\Ec \oplus \Eu$ integrate into $f$-invariant foliations we say that $f$ is \emph{dynamically coherent}. In this case, $\Ec$ integrates into an $f$-invariant foliation obtained by intersection. We will sometimes consider a lift $\tilde f$ to $\tilde M$, the universal cover of $M$ and we shall use $\tilde X$ to denote the lift of $X$ whatever $X$ is.  For $Y$ contained in a metric space $Z$ we define $B_\eps(Y):=\{ z \in Z \ : \ d(z,Y) < \eps \}$. 

We refer the reader to \cite{HP-Survey} and references therein for a more detailed account. In particular, for a proof of the following fact from \cite{BI} that we will use repeatedly:

\begin{prop}\label{p-volvslength}
For every $\eps>0$, there exists $C>0$ such that if $I$ is an arc of $\tFu$ or $\tFs$ then $$\mathrm{volume}( B_\eps( I)) \geq C \cdot \mathrm{length}(I) .$$
\end{prop} 

The proof of this proposition relies heavily on the construction of branching foliations performed in \cite{BI} which we review next.

\subsection{Branching foliations}
In this subsection we introduce the concept of branching foliations and summarise the results from \cite{BI} that we will use. 

A \emph{branching foliation} $\cF$ on $M$ tangent to a distribution $E$ is a collection of immersed surfaces tangent to $E$ so that:

\begin{itemize}
\item every $L \in \cF$ is an orientable, boundaryless and complete (with the induced riemannian metric) surface, 
\item every $x\in M$ belongs to at least one $L \in \cF$,
\item no two surfaces of $\cF$ topologically cross each other, 
\item it is invariant under every diffeomorphism of $M$ whose derivative preserves $E$ and a transverse orientation,
\item if $x_n \to x$ and $L_n$ is a leaf of $\cF$ containing $x_n$ then $L_n \to L$ (uniformly in compact sets) and $x\in L \in \cF$. 
\end{itemize}

We will use the following theorem which combines \cite[Theorem 4.1]{BI} and \cite[Lemma 7.1]{BI} (see also \cite[Section 4]{HP-Survey}): 

\begin{teo}[Burago-Ivanov]\label{teoBI1} Let $f: M \to M$ be a partially hyperbolic diffeomorphism such that $\Es,\Ec,\Eu$ are orientable and their orientation are preserved by $Df$. Then, there exist branching foliations $\Fcs$ and $\Fcu$ tangent respectively to $\Es \oplus \Ec$ and $\Ec \oplus \Eu$. 
\end{teo}

\begin{remark}
The diffeomorphism is dynamically coherent if and only if $\Fcs$ and $\Fcu$ have no branching: each point belongs to a unique surface of the collection (see \cite{HP-Survey}).
\end{remark}

To apply foliation theory, we will need the following result (\cite[Theorem 7.2]{BI}) which is also one of the key points behind Proposition \ref{p-volvslength}. A symmetric statement holds for $\Fcu$.

\begin{teo}[Burago-Ivanov]\label{teoBI2} 
For every $\eps>0$ there exists a foliation $\cW_\eps$ tangent to a distribution $\eps$-close to $\Es \oplus \Ec$ and a continuous map $h_\eps: M \to M$ such that $d(h_\eps(x),x)< \eps$ for all $x\in M$ and such that when restricted to a leaf of $\cW_\eps$ the map $h_\eps$ is $C^1$ onto a leaf of $\Fcs$. 
\end{teo}

Indeed one has (see \cite{HP-Survey} and references therein): %

\begin{fact}\label{fact-diffeo}
The lift $\tilde h_\eps$ of $h_\eps$ to $\tilde M$ when restricted to a leaf of $\tilde \cW_\eps$ is a diffeomorphism onto its image (a leaf of $\tilde \Fcs$). Moreover, these diffeomorphisms vary continuously as one changes the leaf. 
\end{fact}

\begin{fact}\label{fact-reebless}
The foliation $\cW_\eps$ is Reebless\footnote{See also \cite[Section 5]{HP-Survey} for more details on how this follows from Novikov's theorem.} for all small $\eps$ (\cite{CandelConlon}). In particular, when lifted to the universal cover, the space of leaves $\mathcal{L}_\eps$ of the foliation $\tilde \cW_\eps$ is a 1-dimensional, possibly non-Hausdorff simply connected manifold. Using the maps $h_\eps$ and its lift $\tilde h_\eps$ to the universal cover one deduces that the space of leaves of $\tilde \Fcs$ is also a $1$-dimensional, possibly non-Hausdorff simply connected manifold.  
\end{fact}

\subsection{Torus leaves}
In this section we review a result of \cite{HHU-tori} showing that in our context the branching foliations $\Fcs$ and $\Fcu$ cannot have torus leaves. Notice that there exists an example in $\mathbb{T}^3$ where such leaves do exist \cite{HHU-example}. The following is a simplified version of the main result from \cite{HHU-tori} which is enough for our purposes. 

\begin{teo}[Rodriguez Hertz-Rodriguez Hertz-Ures]\label{teoHHU} 
If a partially hyperbolic diffeomorphism $f:M \to M$ has a torus $T$ tangent to $\Es \oplus \Ec$, then, $M$ fibers over $S^1$ with torus fibers. 
\end{teo}

Notice that this implies that if $M$ does not fiber over $S^1$ with torus fibers then neither the branching foliation $\Fcs$ from Theorem \ref{teoBI1} nor the approximating foliations $\cW_\eps$ given by Theorem \ref{teoBI2} can have torus leaves.

\subsection{Circle bundles}\label{ss.circlebundles}
The paper will first deal with the circle bundle case for which there are some simplifications in the notations and for which some results are easier to state. Later in section \ref{s.Seifert} we shall remove this unnecessary hypothesis and work in general Seifert fiber spaces. 

Recall that a circle bundle is a manifold $M$ admitting a smooth map $p: M \to \Sigma$ (in our case, $\Sigma$ is a higher genus surface) such that the fiber $p^{-1}(\{x\})$ for every $x\in \Sigma$ is a circle and there is a local trivialization: for every $x\in \Sigma$ there is an open set $U$ such that $p^{-1}(U)$ is diffeomorphic to $U \times S^1$ via a diffeomorphism preserving the fibers.   

Notice that the fundamental group of the fibers fits into an exact sequence which is a central extension of $\pi_1(\Sigma)$ by $\mathbb{Z}$: 

$$ 0 \rightarrow \pi_1(S^1) \cong \mathbb{Z} \rightarrow  \pi_1(M) \rightarrow \pi_1(\Sigma) \rightarrow 0 $$

Recalling that, if the genus of $\Sigma$ is $g\geq 2$, then the fundamental group $\pi_1(\Sigma)$ admits a presentation:

$$ \pi_1(\Sigma) = \bigg\langle a_1,b_1, \ldots, a_g, b_g  \  | \ \prod_{i=1}^g [a_i,b_i] = \mathrm{id} \bigg\rangle,  $$

\noindent one obtains that $\pi_1(M)$ admits a presentation:

$$ \pi_1(M) = \bigg\langle a_1,b_1, \ldots, a_g, b_g, c \  | \  \prod_{i=1}^g [a_i,b_i] = c^{\mathrm{eu}(M)} \ , \ [a_i,c]= \mathrm{id} \ , \ [b_i,c]=\mathrm{id} \bigg\rangle,  $$

\noindent where $\mathrm{eu}(M)$ is the \emph{Euler number} of the circle bundle. See \cite[Book II, Chapter 4]{CandelConlon} for more details on the Euler number and the fundamental group of $M$. Let us just end by noticing that the \emph{center} of $\pi_1(M)$ is the (cyclic) group generated by the fundamental group of the fiber, called $c$ in the above presentation. As the center is a group invariant and in this case is isomorphic to $\mathbb{Z}$ one knows that every automorphism of $\pi_1(M)$ will send $c$ to either $c$ or $c^{-1}$. 

Finally, let us point out that circle bundles over surfaces of genus $\geq 2$ do not fiber over $S^1$ with torus fibers so that  Theorem \ref{teoHHU} will imply that a partially hyperbolic diffeomorphism on such an $M$ cannot have a torus tangent to $\Es \oplus \Ec$ or $\Ec \oplus \Eu$. 

\subsection{Taut foliations on circle bundles} 

Let $M$ be a circle bundle and let $\cF$ be a foliation on $M$ which has no torus leaves. We remark that we use the convention here that a foliation is a $C^{0,1+}$-foliation in the sense that it has only continuous trivialization charts, but leaves are $C^1$ and tangent to a continuous distribution (see \cite{CandelConlon}). 

Foliations without torus leaves in 3-manifolds are examples of what is known as \emph{taut foliations}. We will state a result of Brittenham \cite{Brittenham} as it appears in \cite[\S4.10]{Calegari} (where a different proof in the case of circle bundles can be found). We remark that Brittenham result is an extension of an old result of Thurston that was also improved by Levitt (see \cite{Levitt-novert}). The proof of \cite{Levitt-novert} uses the hypothesis that $\cF$ is $C^2$ but most of the proof (except the part where it is shown that there are no vertical leaves) can be applied in the $C^{0,1+}$ case as they only use results of putting tori in general position which are now available for $C^0$-foliations \cite{Solodov} (see also the proof of \cite[Proposition 2.3]{Ghys}).  

We use the following definitions.
A leaf $L$ of $\cF$ is \emph{vertical} if it contains every fiber it intersects. Equivalently, there exists a properly immersed curve $\gamma$ in $\Sigma$ such that $L= p^{-1}(\gamma)$. 
A leaf $L$ of $\cF$ is \emph{horizontal} if it is transverse to the fibers (this includes the possibility of not intersecting some of them). 

\begin{teo}[Brittenham-Thurston]\label{teo-brit}
Let $\cF$ be a foliation without torus leaves in a Seifert fiber space $M$. Then, there is an isotopy $\psi_t: M \to M$ from the identity such that the foliation $\psi_1(\cF)$ verifies that every leaf is either everywhere transverse to the fibers (horizontal) or saturated by fibers (vertical).   
\end{teo}

Of course, it is possible to work inversely and apply the isotopy to $p$ in order to have that $\cF$ has the property that every leaf is vertical or horizontal with respect to $p\circ \psi_1$. 

\begin{remark}
After Theorem \ref{teo-brit} it makes sense, for a taut foliation $\cF$, to call a leaf vertical if it has a loop freely homotopic to a fiber, and horizontal if it has no such loop. We will adopt this point of view in what follows. 
\end{remark}

\section{Vertical leaves}\label{s.novert}

Let $M$ be an orientable circle bundle over an orientable surface $\Sigma$ of genus $g\geq 2$. Let $p: M \to \Sigma$ denote the projection of this circle bundle. 

We consider $f: M \to M$ to be a partially hyperbolic diffeomorphism such that the bundles $\Es,\Ec,\Eu$ are orientable and its orientation is preserved by $Df$. 

Let $\Fcs,\Fcu$ be fixed branching foliations associated to $f$. 

We will now say that a leaf $L$ of $\Fcs$ or $\Fcu$ is \emph{vertical} if its fundamental group contains an element of the center of $\pi_1(M)$ (i.e., it contains a loop freely homotopic in $M$ to a fiber of the circle bundle). See the remark after Theorem \ref{teo-brit}. 

We state the following result which will be proved in the next sections. 

\begin{thm}\label{teo-novert} 
For $f$ and $M$ as above, under any of the following assumptions, the branching foliations $\Fcs$ and $\Fcu$ have no vertical leaves:
\begin{enumerate}
\item \label{it-transitivity} $f$ is \emph{chain recurrent}\footnote{This includes as particular cases when $f$ is transitive or volume preserving.}, i.e., if $U$ is a non-empty open set such that $f(\overline U)\subset U$ then $U=M$;
\item \label{it-coherence} $f$ is \emph{dynamically coherent}, i.e., the branching foliations have no branching;
\item \label{it-isotopicid} $f$ is \emph{homotopic to the identity};
\item \label{it-isotopicpA} the action of $f$ on $\pi_1(\Sigma)$ is \emph{pseudo-Anosov};
\item \label{it-absolute} $f$ is \emph{absolutely partially hyperbolic}.  
\end{enumerate}
\end{thm}

In item (\ref{it-isotopicpA}) we are taking into account that $f_\ast : \pi_1(M) \to \pi_1(M)$ preserves the generator corresponding to the fibers, and therefore induces an action on $\pi_1(\Sigma)$ via the fiber bundle projection $p$.  
The general case, where $f$ may not be dynamically coherent and the action on $\pi_1(\Sigma)$ is reducible is still open, though we think that the proofs we provide shed some light on how to attack it.  Notice that recently non-dynamically coherent examples have appeared in Seifert manifolds (\cite{BGHP}) but these have horizontal branching foliations. %

\medskip

\begin{remark} To prove item (\ref{it-coherence}) we give a general statement about vertical $cs$ and $cu$-laminations which works even if the diffeomorphism is not dynamically coherent (see Proposition \ref{p.novertifcircle}).  Similarly, to prove point (\ref{it-isotopicpA}) we show a slightly more general statement that may be useful in certain situations (see Proposition \ref{p-volvslengthsmall}).
\end{remark}

Using Theorems \ref{teo-brit} and \ref{teoBI2} we get the following result which is what we will use to obtain our main theorem:

\begin{cor}\label{cor-novert} Under any of the conditions (\ref{it-transitivity})-(\ref{it-absolute}) it follows that for every small $\eps>0$ there exists a projection $p_\eps$ isotopic to $p$ such that for $\cW_\eps$, the approximating foliation to $\Fcs$, one has that every leaf is transverse to the fibers of $p_\eps$. 
\end{cor}

The proofs are symmetric on $cs$ and $cu$, so we concentrate on the $cs$-case. Item (\ref{it-coherence}) is the most involved and will be proved separately in section \ref{sec-coherent}. The rest of the items will be proved in section \ref{sec-novertproof} and their proofs are independent of section \ref{sec-coherent} with the exception of item (\ref{it-absolute}). 

\subsection{Generalities} 
Define $\Lambda^{cs} \subset M$ to be the union of all vertical $cs$-leaves. 

\begin{prop}\label{prop-Lambdacompactinv} 
The set $\Lambda^{cs}$ is compact, $f$-invariant and $\Lambda^{cs}\neq M$. 
\end{prop}
\begin{proof} The action of $f$ on the fundamental group must preserve the center of $\pi_1(M)$ and therefore, as leaves of $\Fcs$ are sent to leaves of $\Fcs$, those which contain a loop homotopic to a fiber are invariant.  

To show that $\Lambda^{cs}$ is closed, we work in the universal cover, and notice that if a sequence $x_n \to x$ and $L_n$ are leaves through $x_n$ which are invariant under the deck transformation generated by $c$ (a generator of the center of $\pi_1(M)$) then, the same holds for a leaf $L$ which is the limit (in the sense of uniform convergence in compact sets). Therefore, $L$ also belongs to the lift of $\Lambda^{cs}$ and contains $x$. %

To show that $\Lambda^{cs} \neq M$ we will make an \emph{ad hoc} direct argument\footnote{Notice that if $f$ is dynamically coherent, this follows directly by applying Theorem \ref{teo-brit}. Indeed, as $\Lambda^{cs}$ would be a vertical foliation one would be able to project the leaves to get a one dimensional foliation in the surface, contradicting that it has genus $\geq 2$. When there is branching, one cannot apply this reasoning directly because the foliation obtained by blowing up in Theorem \ref{teoBI2} could create some non-vertical leaves even if every leaf of the branching foliation is vertical.}. An alternative proof can be obtained by using the rest of the results in this section. 

We follow an argument borrowed from \cite{Barbot2}. We work in the universal cover, where the properties of the branching foliation imply that the space of leaves $\mathcal{L}$ is a $1$-dimensional (possibly non-Hausdorff) simply connected manifold  (as explained in Fact \ref{fact-reebless}) where $\pi_1(M)$ acts. %

If one assumes by contradiction that $\Lambda^{cs}=M$, one obtains that every leaf is fixed by  a generator $c$ of the center of $\pi_1(M)$. As there is no torus tangent to $E^s \oplus E^c$ and $c$ commutes with every other element of $\pi_1(M)$, it follows that every element not in the center of $\pi_1(M)$ acts freely on $\mathcal{L}$. This implies, by a Theorem of Sacksteder (see \cite[Theorem 3.3]{Barbot2}) that $\mathcal{L}$ is a line and the action is semiconjugate to an action by translations. 

Moreover, if one chooses two elements $g_1,g_2$ in $\pi_1(M)$ such that neither $g_1$, $g_2$ nor its commutator belong to the center of $\pi_1(M)$, one obtains a contradiction since the commutator should be semiconjugate to identity, but it cannot have fixed points in $\mathcal{L}$.  %
\end{proof}

The reader only interested in item (\ref{it-transitivity}) of Theorem \ref{teo-novert} can safely skip to subsection \ref{ss.trans} for a direct proof. 

To continue we will introduce the concept of \emph{quasi-isometrically embedded submanifold.} Given $b>0$ we say that a submanifold $L$ of a manifold $M$ is $b$-\emph{quasi-isometric} if for every $x,y \in L$ one has that:

$$ d_L (x,y) \leq b d_M (x,y) + b,  $$

\noindent where $d_L$ denotes the metric induced by $M$ on $L$ and $d_M$ is the metric in $M$. 

\begin{prop}\label{prop-quasigeodesic}
   Let $\Gamma < \mathrm{PSL}(2,\mathbb{R})$ be the set of deck transformations for a given compact surface $\Sigma$ and consider $\mathcal{A}=\{\eta_i\}_{i \in I}$ a family of properly embedded copies of  $\ \mathbb{R}$ with the following properties: 
   \begin{itemize}
   \item for every $i \in I$ and $g \in \Gamma$ one has that $g \eta_i \in \mathcal{A}$;
   \item for every $i,j \in I$ (not necessarily different) and $g \in \Gamma$, the curves $g \eta_i$ and $\eta_j$ do not have topologically transverse intersections; 
   \item if $\eta_{i_k} \in \mathcal{A}$ is a sequence such that $\eta_{i_k} \to \eta$ uniformly in compact sets, then $\eta \in \mathcal{A}$. 
   \end{itemize}
   Then, there exists $b>0$ such that every curve in $\mathcal{A}$ is $b$-quasi-isometric. 
 \end{prop}

The proof of this proposition is purely geometric and is relegated to an appendix. Similar results exist in the literature, see e.g. \cite{Levitt} or \cite{Calegari}.

We use this to provide some basic facts about vertical $cs$-laminations. We will say that $\Lambda \subset \Lambda^{cs}$ is a \emph{sublamination} if it is closed, $f$-invariant and saturated by $cs$-leaves. We shall call a leaf $L$ an \emph{accessible boundary leaf} of $\Lam$  if it contains an endpoint of a segment whose interior is contained in $M\setminus \Lambda$. For a sublamination $\Lambda \subset \Lambda^{cs}$, let $\tilde \Lambda$ denote the set of lifts of leaves in $\Lambda$ to $\tilde M$. 

\begin{lemma}\label{l.verticallamination} 
Let $\Lambda \subset \Lambda^{cs}$ be a  non-empty sublamination. Then: 
\begin{enumerate}
\item\label{it-quasiisomleaves} there exists $b>0$ such that the leaves of $\tilde \Lambda$ are $b$-quasi-isometric in $\tilde M$;
\item\label{it-polynomialarea} the leaves of $\tilde \Lambda$ are parabolic, i.e. they have polynomial growth of area of balls with respect to the radius;
\item\label{it-periodic} there exists a connected component $U$ of $M \setminus \Lam$ which is periodic and the leaves of $\Lam$ in the (accessible) boundary of $U$ are periodic;
\item\label{it-unst} every connected component $U$ of $M\setminus \Lam$ which is periodic contains an unstable arc whose forward iterates remain far from $\Lambda$.
\end{enumerate}
\end{lemma}

We will prove this lemma by applying Theorems \ref{teoBI2} and \ref{teo-brit} and using Proposition \ref{prop-quasigeodesic}. %

\begin{proof} For the proof of items (\ref{it-quasiisomleaves}) and (\ref{it-polynomialarea}), let $\cW_\eps$ be an approximating foliation to $\Fcs$ with small $\eps$ given by Theorem \ref{teoBI2}. As the map $h_\eps$ is close to the identity, for every leaf $L \in \Lam^{cs}$ there is a leaf of $\cW_\eps$ having a loop freely homotopic to a fiber which is mapped onto $L$ by $h_\eps$. Therefore, the vertical sublamination of $\cW_\eps$ is mapped onto $\Lam^{cs}$. Applying Theorem \ref{teo-brit} to $\cW_\eps$ one obtains a sublamination which is mapped to $\Lam^{cs}$ made of vertical leaves (up to changing the fibration by a homotopy).%

We will use that $\tilde h_\eps$ in the universal cover is close to the identity and is a diffeomorphism when restricted to each leaf of $\tilde \cW_\eps$ (see Fact \ref{fact-diffeo}). In particular, it is enough to show items (\ref{it-quasiisomleaves}) and (\ref{it-polynomialarea}) for vertical leaves of $\tilde \cW_{\eps}$. 

To show (\ref{it-quasiisomleaves}) one needs to notice that the projection of the vertical leaves by $p_\eps$ is a family of curves in the surface that are in the hypothesis of Proposition \ref{prop-quasigeodesic}. As all fibers have bounded length, the result follows.  %

To show  (\ref{it-polynomialarea}) one has to notice that the length of the fibers is uniformly bounded and therefore the volume of the ball of radius $R$ of a vertical leaf in the universal cover is quadratic in $R$. (See also the Remark after \cite[Theorem 4.56]{Calegari}.)

We turn now to items (\ref{it-periodic}) and (\ref{it-unst}).

Take an unstable arc $I$ whose interior is in $M\setminus \Lam$ and has one endpoint $x_0 \in \Lam$. As the length of forward iterates of $I$ increases and each connected component of $M\setminus \Lam$ containing an unstable arc of length $\geq \eps$ has volume bounded from below, we know that the connected components of $M \setminus \Lam$ are periodic.  %

Moreover, each periodic component of $M\setminus \Lam$ can have at most finitely many accessible boundary leaves for a similar reason, so one obtains (\ref{it-periodic}).

To complete the proof of (\ref{it-unst}) consider an arc $I$ tangent to $\Eu$ which without loss of generality we assume its interior is contained in a fixed  connected component of $M \setminus \Lam$ and  whose endpoint $x_0$ belongs to a fixed accessible boundary leaf $L$. If the half unstable leaf $W^u_+(x_0)$ (i.e. the connected component of $W^u(x_0)\setminus \{x_0\}$ containing $I$) is not completely contained in $U$ it follows that there is a compact interval $J$ tangent to $\Eu$ in $U$ whose endpoints $x_0$ and $y_0$ belong to accessible boundary leaves. Iterating forward one obtains arbitrarily large unstable arcs contained in $U$ whose distance to the boundary of $U$ is  at least $\eps$ (the size of local product structure boxes). Taking limits, one completes the proof of (\ref{it-unst}). 
\end{proof} 

\section{Dynamically coherent case}\label{sec-coherent} 

In this section, $f: M \to M$ will be a dynamically coherent partially hyperbolic diffeomorphism of a circle bundle $M$ with $f$-invariant foliations $\cW^{cs}$ and $\cW^{cu}$ tangent respectively to $\Ecs$ and $\Ecu$. We will prove item (\ref{it-coherence}) of Theorem \ref{teo-novert} by contradiction, so we assume that there exists $\Lambda^{cs}$ a vertical sublamination of $\cW^{cs}$ to reach a contradiction. As in the previous section, we assume that $\Es,\Ec,\Eu$ are orientable and that their orientations are preserved by $Df$. 

Dynamical coherence will only be used at a few specific points.
In subsection \ref{ss.furthercomments}, we further discuss this usage as it is the subtlest part of the proof of Theorem \ref{teo-novert}. 

\subsection{An invariant annulus}\label{ss.invariantanuli}

Using the orientation of the bundles, define $\Ws_+(x)$ to be the positively oriented connected component of $\Ws(x)\setminus \{x\}$ with the point $x$ added. Similarly, one defines $\Ws_-(x)$, $\Wu_+(x)$ and $\Wu_-(x)$. 

Consider a minimal sublamination $\Lambda \subset \Lambda^{cs}$.  Lemma \ref{l.verticallamination} item (\ref{it-periodic}) implies that there exists a periodic component $U$ whose boundary leaves are also periodic. Without loss of generality we will assume that $f(U)=U$ and $f(L) = L$ for $L$ a boundary leaf of $\Lambda$. 

For $x \in L$ we define $J_x$ to be the closed segment of $\Wu(x)$
whose interior is contained in $U$ and such that one endpoint is $x$
and the other endpoint (if it exists)
belongs to $\Lambda$.
We can without loss of generality assume that $J_x \subset \Wu_+(x)$.
Let $\ell_x \in (0,\infty]$ denote the length of $J_x$. Clearly, $\ell_x = \infty$ means that $\Wu_+(x) \setminus \{x\} \subset U$. 

We show here: 

\begin{lemma}\label{lema-anulus}
There exists an essential closed annulus $X \subset L$ such that $f^k(X)$ is contained in the interior of $X$ for some $k>0$.
\end{lemma}

\begin{proof} 
By a volume argument,
one can show that the set
\begin{math}
    \{ x \in L: \ell_x  \ge  1 \}
\end{math}
is compact in the topology of $L$.
Let $X \subset L$ be a compact annulus
such that $\ell_y < 1$ for all $y \notin X$.

There is $\delta > 0$ such that $\ell_x  \ge  \delta$
for all $x \in X$.
By the expansion of $\Eu$, there is $k  \ge  1$ such that
$\ell_{f^k(x)} > 1$ for all $x \in X$.
This implies that $f^k(X) \subset \mathrm{int} (X)$.
\end{proof}

\subsection{Finding a circle}\label{ss.findcircle}
Using that $f$ is dynamically coherent, we can define a center foliation in $L$ by intersecting it with the foliation $\cW^{cu}$ and choosing the connected components of the intersections. Let $\cF^c$ denote this center foliation. The purpose of this subsection is to show that there is a fixed circle leaf of $\cF^c$. A $\sigma$-curve will be an arc in a leaf of $\cF^\sigma$ with $\sigma=s,c,u$. As the bundles are all oriented, one can choose an orientation on $\sigma$-arcs. A concatenation of $s$ and $c$-curves is \emph{coherently oriented} if one can parametrise the concatenation so that each time one considers a $c$-curve the parametrisation has the same orientation as $\Ec$ (see figure \ref{f.coherent}). 

\begin{figure}[ht]
\vspace{-0.5cm}
\begin{center}
\includegraphics[scale=0.6]{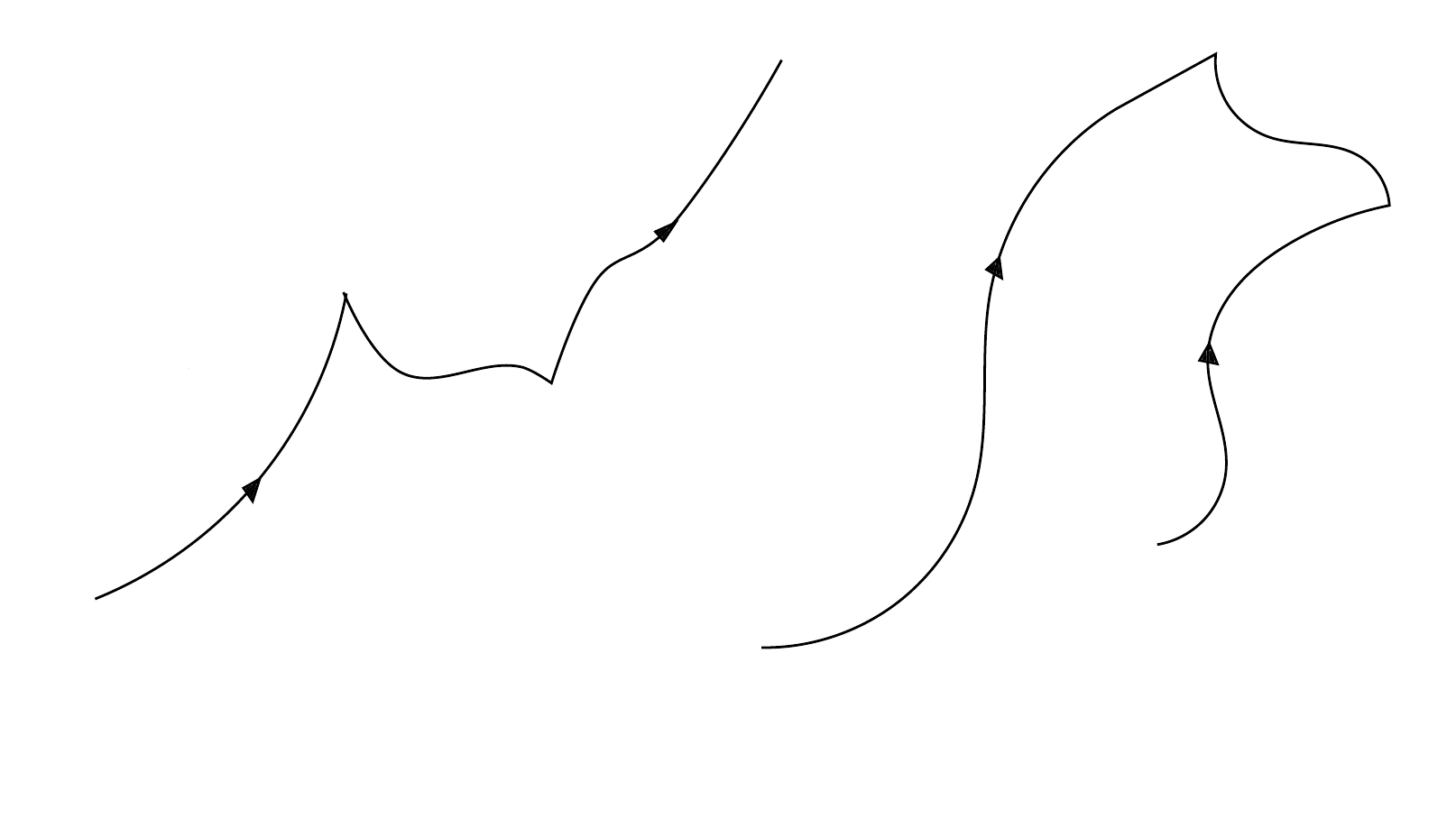}
\begin{picture}(0,0)
\put(-34,135){$s$}
\put(-45,95){$c$}
\put(-89,95){$c$}
\put(-166,95){$c$}
\put(-197,90){$s$}
\put(-245,60){$c$}
\end{picture}
\end{center}
\vspace{-0.5cm}
\caption{The concatenation on the left is coherently oriented and the one on the right is not.\label{f.coherent}}
\end{figure}
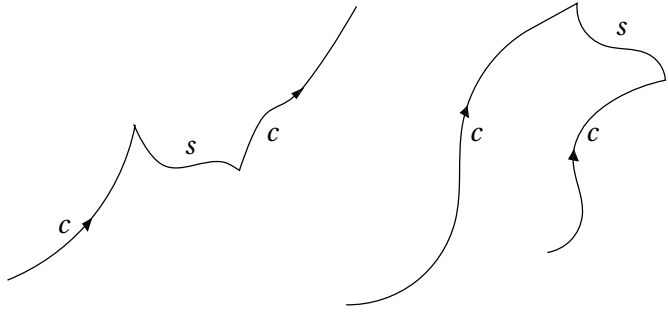

We will need the following consequence of a graph transform argument which we state without proof (see \cite{HPS} for similar arguments):

\begin{lemma}\label{l.graphtransf}
Consider a closed curve $\eta$ contained in a compact forward invariant subset of a fixed $cs$-leaf $L$ with the following properties:
\begin{itemize}
\item it is made from finitely many alternated $c$ and $s$-curves;
\item the curves are coherently oriented;
\item the length of forward iterates is bounded.
\end{itemize}
Then, the forward iterates of $\eta$ converge to a fixed circle $c$-leaf $\gamma$ in $L$.
\end{lemma}

We use this lemma to establish the existence of a circle leaf. It is instructive to see how this proposition fails in the non-dynamically coherent example of \cite{HHU-example}.

\begin{prop}\label{prop-circleexist}
There exists a circle leaf $\gamma$ of $\cF^{c}$ in $L$ which is fixed by $f$. 
\end{prop}

\begin{proof}
As the annulus $X$ given by Lemma \ref{lema-anulus} is compact, one can cover it by finitely many local product structure boxes. As there cannot be circle $s$-leaves, one has that there exists $\ell_0>0$ so that every $s$-curve of length $\geq \ell_0$ is not entirely contained in $X$. 

\begin{af}
Either there is a circle leaf  of $\cF^c$ fixed by $f$ in $X$, or there exists $\ell_1>0$ so that every $c$-curve of length $\geq \ell_1$ is not entirely contained in $X$.
\end{af}

\begin{proof}[Proof of claim.]
Cover $X$ with local product structure boxes. As $X$ is compact, it follows that it can be covered with finitely many such boxes. We will show that if a $c$-curve intersects a box twice, then one can apply Lemma \ref{l.graphtransf} to obtain a closed curve $\gamma$ of $\cF^c$ which is fixed by $f$. If each $c$-curve intersects each box at most once, then it is easy to see that there exists $\ell_1>0$ as in the conclusion of the claim. 

Indeed, consider a curve $\hat \eta \subset X$ intersecting the same local product structure box twice. By joining the endpoints with a stable curve, one can construct a coherently oriented closed curve $\eta$ and so to apply Lemma \ref{l.graphtransf} it is enough to show that the length of $f^n(\eta)$ remains bounded.

Assume that the length of $f^n(\eta)$ is not bounded, then, there must be some $n\geq 1$ and a local product structure box which $f^n(\eta)$ intersects at least three times. As $X$ is an annulus and $f^n(\eta)$ has no self-intersections, it follows that there must be some center arc inside $f^n(\eta)$ which can be closed by a stable curve into a nullhomotopic curve. This gives a contradiction
and proves the claim.
\end{proof}

The rest of this proof is devoted to showing that the second possibility
in the above claim actually implies the first one.  

    Using the finite cover of $X$ by local product structure boxes, one can construct a circle $\eta$ in $X$ made up of finitely many alternated $c$-curves $\alpha_i$ and $s$-curves $\beta_i$.  In principle, the orientation of the $c$-curves $\alpha_i$ may differ, but we will show that, under the assumption of dynamical coherence, one can remove center curves until all of them are oriented in the same way (and then Lemma \ref{l.graphtransf} gives the desired circle). Clearly, the iterates of $\eta$ have bounded length as they remain in $X$ and the $c$-curves and $s$-curves have uniformly bounded length. 

The important point of dynamical coherence is that we can ensure that for every $\eps>0$ there exists $\delta>0$ such that if two points $x_0,y_0$ are in the same stable leaf at distance less than $\delta$, then their $c$-curves of length $\leq \ell_1$ remain at distance less than $\eps$. 

Assume then that $\alpha_i$ and $\alpha_{i+1}$ have different orientations and $\beta_i$ is a stable arc joining the endpoint $x_0$ of $\alpha_i$ and the start point $y_0$ of $\alpha_{i+1}$. By iterating forward and using the contraction of $\beta_i$ we may assume that $\beta_i$ has length smaller than $\delta$, and therefore the arcs $\alpha_i$ and $\alpha_{i+1}$, being oriented differently fit in a map $\phi: [0,1]^2 \to X$ such that:

\begin{itemize}
\item $\phi(0,0) = x_0$ and $\phi(0,1)=y_0$, and $\phi( \{0\} \times [0,1])=\beta_i$,
\item the arcs $\phi(\{t\} \times [0,1])$ are stable and of length smaller than $\eps$,
\item either $\phi([0,1]\times \{0\}) = \alpha_i$ (with reversed orientation) or $\phi([0,1]\times \{1\}) = \alpha_{i+1}$ (with the correct orientation). 
\end{itemize}

\begin{figure}[ht]
\vspace{-0.5cm}
\begin{center}
\includegraphics[scale=0.6]{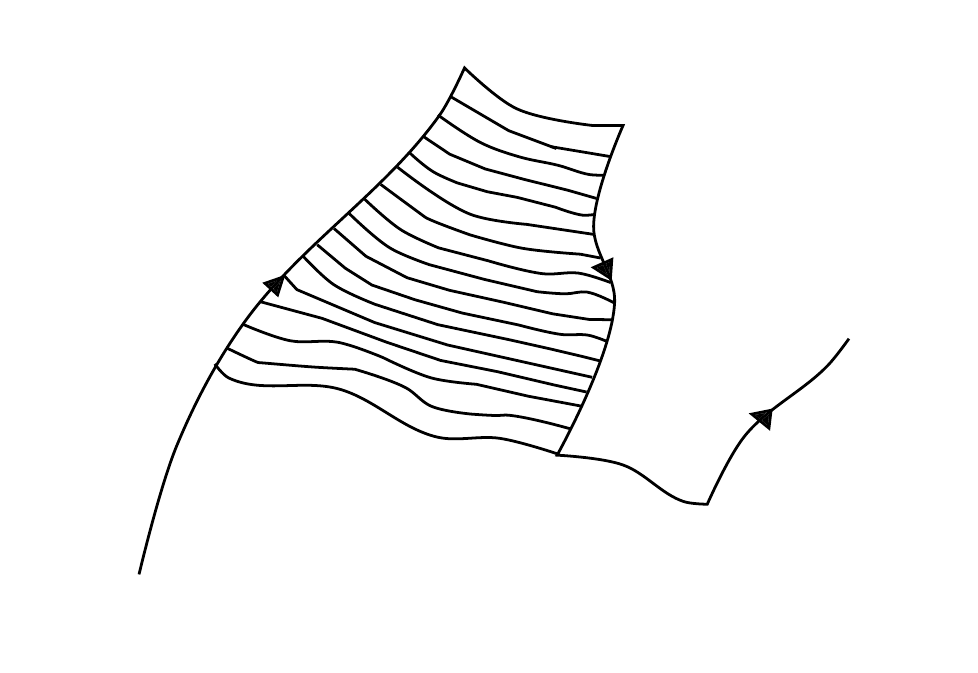}
\begin{picture}(0,0)
\put(-79,100){$s$}
\put(-79,30){$s$}
\put(-145,60){$\alpha_i$}
\put(-65,60){$\alpha_{i+1}$}
\put(-30,45){$\alpha_{i+2}$}
\end{picture}
\end{center}
\vspace{-0.5cm}
\caption{Cutting a curve in the concatenation.\label{f.cutcurve}}
\end{figure}
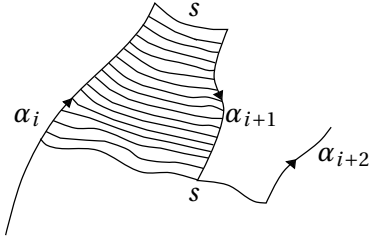

Without loss of generality (by changing slightly the choices) one can assume that in the last item it is not possible that both possibilities occur, i.e. by choosing the initial curve so that no pair of endpoints of $c$-curves belong to the same stable leaf. 

Therefore, one can remove one of the center arcs (say, $\alpha_i$ if $\phi([0,1]\times \{0\}) = \alpha_i$) and cut $\alpha_{i+1}$ by eliminating $\beta_i$ and changing it from the stable arc made by the concatenation of $\beta_{i-1}$ and $\phi(\{1\} \times [0,1])$ which joins an endpoint of $\alpha_{i-1}$ and an endpoint of the new $\alpha_{i+1}$. This procedure allows one to reduce the number of $c$-curves whenever two consecutive curves are not coherently oriented, and therefore, after finitely many steps one finds a circle made up of alternated $s$ and $c$-curves all of which are coherently oriented.
Lemma \ref{l.graphtransf} then yields the desired fixed circle leaf.
\end{proof}

\subsection{No vertical leaves}\label{ss.finishproof}
Proposition \ref{prop-circleexist} provides a circle leaf $\gamma$ in $L$ which is fixed by $f$. By construction, this circle belongs to a leaf $F$ of $\cW^{cu}$ which is also fixed by $f$. 

\begin{lemma}\label{lem.othercircle}
There exists a circle leaf $\alpha$ of $\cF^c$ in $L$ distinct from $\gamma$. 
\end{lemma}

\begin{proof} As the leaf $L$ belongs to a minimal sublamination, $L$ must accumulate on itself. As the leaves $L$ and $F$ are both vertical, the connected components of their intersections are compact and therefore are circles. Since $F$ intersects $L$ at $\gamma$, it will intersect in a different circle $\alpha \neq \gamma$ when $L$ accumulates on itself. 
\end{proof}

Let $Y$ and $Z$ be the connected components of $L \setminus \gamma$. We assume that the curve $\alpha$ given by Lemma \ref{lem.othercircle} is contained in $Y$. Without loss of generality we may assume that $Y$ contains $\Ws_+(y)\setminus \{y\}$ for some $y \in \gamma$.

Define $\Wsp(\gamma) = \bigcup_{x  \in  \gam}$ $\Wsp(x)$
and define $V \subof L$ as the open annulus between $\gamma$ and $\alpha$.

Using the ideas of \cite[Section 2.1]{BonattiWilkinson} one gets:

\begin{lemma} \label{lemma:Yequals}
    The equality
    \, $Y \, = \, \Wsp(\gamma) \setminus \gamma \, = \,
    \bigcup_{n  \in  \bbZ} f^n(V)$ \, holds.
\end{lemma}
\begin{proof}
    The stable foliation $\Ws$ on $L$ has no compact leaves
    and the curves $\alpha$ and $\gamma$ are both transverse to $\Ws$.
    By a Poincar\'e-Bendixson argument,
    one may use these properties to
    show that the annulus $V$ is trivially foliated
    by stable segments.
    In other words, both $V$ and $\alpha$ are subsets of 
    $\Wsp(\gamma) \setminus \gamma$, and so
    \[
        \bigcup_{n  \in  \bbZ} f^n(V)
        \, \subof \,
        \Wsp(\gamma) \setminus \gamma
        \, \subof \,
        Y.
    \]
    Up to replacing $f$ by an iterate,
    we may assume that $f(\alpha)$ is contained in $V$.

    If $\bigcup_{n  \in  \bbZ} f^n(V)$ is a proper subset of $Y$,
    then there is a point $y  \in  Y$ lying on its boundary.
    In particular,
    \[
    \dist(y, f \invn(\alpha)) = \dist(y,  f \invn(V))  \to  0
    \]
    as $n  \to  \infty$.
    By considering the local product structure near $y$,
    there is, for each large $n$, a small stable segment $I_n \subof \Ws(y)$
    which has one endpoint on $f \invn(\alpha)$ and the other on $f^{-n+1}(\alpha)$.
    Then, $\{f^n(I_n)\}$ is a sequence of stable curves, each connecting
    $\alpha$ to $f(\alpha)$ and whose lengths tend to zero.
    As $\alpha$ and $f(\alpha)$ are disjoint, this gives a contradiction.
\end{proof}

\begin{cor} \label{cor:alconverge}
    The curves $f^n(\alpha)$ converge uniformly to $\gamma$ as $n  \to  \infty$.
\end{cor}
Define $\Wsp(\alpha) = \bigcup_{x  \in  \alpha}$ $\Wsp(x)$
and note that $V = Y \sans \Wsp(\alpha)$.

Recall the notation $J_y$ and $\ell_y$ defined above.

\begin{lemma} \label{lemma:lfinite}
    The length $\ell_y$ is uniformly bounded for $y  \in  \Wsp(\alpha)$.
\end{lemma}
\begin{proof}
    As the set $X$ in the proof of Lemma \ref{lema-anulus} is compact,
    there is $n  \in  \bbZ$ such that $X \cap Y \subof f^n(V)$.
    This shows that $\ell_y  \le  1$ for $y  \in   f^n(\Wsp(\alpha))$.
\end{proof}

We now consider unstable holonomies.
For $\eps > 0$, define the local stable manifold
\begin{math}
    \Ws_\eps(x) = \{ y  \in  \Ws(x) : d_s(x,y) < \eps \}.
\end{math}

\begin{lemma} \label{lemma:holonomy}
    For any $y  \in  L$ and $z  \in  \Wu(y)$, there is $\eps > 0$
    and a continuous map $h^u : \Ws_\eps(z) \to L$
    such that $h^u(x)  \in  \Wu(x)$ for all $x  \in  \Ws_\eps(z)$.

    Moreover, for any $C > 0$, there is a uniform value of $\eps > 0$
    such that the above holds for all points with
    $d_u(y,z) < C$.
\end{lemma}
\begin{proof}
    This follows from standard properties of foliations
    and the fact that $\Wu$ is uniformly transverse to $L$.
\end{proof}

Define an annulus $A = \bigcup_{x  \in  \alpha} J_x$.

\begin{lemma} \label{lemma:tail}
    Any stable leaf intersects $A$ in at most one point.
          \end{lemma}
\begin{proof}
    Suppose $y  \in  A$ and $z  \in  \Wsp(y) \cap A$.
    Then let $[y,z]$
    represent the compact stable segment between the two points.
    Without loss of generality,
    assume no other points in $[y,z]$ intersect $A$.
    That is, $[y,z] \cap A = \{y,z\}$.

    Let $I$ be the largest connected subset of $[y,z]$
    such that $y  \in  I$ and for which there is a well-defined
    unstable holonomy $h^u$ : $I  \to  L$.
    If $x  \in  I \sans \{y,z\}$, then $x \notin A$ and therefore $h^u(x) \notin \alpha$.
    From this, one may show that $h^u(I) \subof \Wsp(\alpha)$.
    By the previous two lemmas, there is a fixed $\eps > 0$
    such that if $x  \in  I$, then $\Ws_\eps(x) \cap [y,z] \subof I$.
    It follows that $I = [y,z]$.

    By considering the orientation of $\Es$, one sees that
    $h^u(I)$ must exit one side of $\alpha$ and return to $\alpha$ on the other side.
    Since the circle $\alpha$ cuts the cylindrical leaf $L$ into two pieces,
    such a path is not possible.
\end{proof}

Define $\Wup(\gamma) = \bigcup_{x \in \gamma} \Wup(x)$.

\begin{cor}
    Any stable leaf intersects $\Wup(\gamma)$
    in at most one point.
\end{cor}
\begin{proof}
    For a point $z  \in  \Wup(\gamma)$,
    the image of the holonomy $h^u$ : $\Ws_\eps(z)  \to  L$
    is a continuous curve $h^u(\Ws_\eps(z))$
    which crosses from one side of $\gam$ to the other.
    By \cref{cor:alconverge}, this curve intersects $f^n(\alpha)$ for all large $n$.
    Consequently, $\Ws_\eps(z)$ intersects $f^n(A)$ for all large $n$.

    If $\Ws(z)$ intersects $\Wup(\gamma)$ in another point, $\hat z$,
    then there are $\eps$ and $n$ such that $f^n(A)$ intersects
    $\Ws_\eps(z)$ and $\Ws_\eps(\hat z)$ in two distinct points,
    and \cref{lemma:tail} gives
    a contradiction.
\end{proof}

 This  contradicts the following general lemma and completes the proof of item (\ref{it-coherence}) of Theorem \ref{teo-novert}.

\begin{lemma}\label{lem-twopoints} 
   There is  $z \in \Wup(\gamma)$ such that $\Wsp(z)$ intersects $\Wup(\gamma)$
    in at least two distinct points.
\end{lemma}

\begin{proof}
    First, if $\Wup(\gamma)$ were complete as a submanifold,
    it would accumulate at a point in $M$, implying the result .
    Therefore, assume it is not complete.

    Look at the holonomy of the center foliation inside $\Wup(\gamma)$
    in a small neighbourhood of $\gamma$.
    Trivial holonomy would imply that $\Wup(\gamma)$ is complete (see Lemma \ref{lemma:Yequals} or \cite{BonattiWilkinson}).
    Therefore, the holonomy is non-trivial and there is a circle
    in $\Wup(\gamma)$ close to $\gamma$ and transverse to the center direction.
    Up to changing the orientation of $\Ec$, we have the following
    property{:}
    if $\hat x \in \Wup(\gamma)$ and $y \in \Wcp(\hat x)$,
    then $y \in \Wup(\gamma)$.

    The incomplete submanifold $\Wup(\gamma)$ accumulates
    at a point $z_0 \notin \Wup(\gamma)$.
    As $\Wup(z_0)$ is not compact, there is a small plaque $P$ in a leaf of the $cs$-foliation
    such that $\Wup(z_0)$ intersects $P$
    in an infinite set.
    We further assume that the center and stable foliations have
    local product structure on $P$.
    There is a convergent sequence of points
    $z_k \in \Wup(z_0) \cap P$
    and, for each $k$, a sequence $y_{k,\ell} \in \Wup(\gamma) \cap P$
    such that $\lim_{\ell \to \infty} y_{k,\ell} = z_k$.

    Case one{:}
    all $y_{k,\ell}$ lie on the same local center leaf.
    Then, all $z_k$ also lie on the same local center leaf, and one can find
    indices $k,\ell,m$ such that $z_m \in \Wcp(y_{k,\ell})$.
    This implies that $z_m \in \Wup(\gamma)$, a contradiction.

    Case two{:}
    there are $y_{k,\ell}$ on different center leaves.
    Call the two points $y$ and $\hat y$.
    Then, either
    $\Wcp(y)$ intersects $\Ws(\hat y)$
    or
    $\Wcp(\hat y)$ intersects $\Ws(y)$.
    In either case, the intersection is in $\Wup(\gamma)$
    and the result is proved.
\end{proof}

\subsection{Further comments on the proof}\label{ss.furthercomments} 
There are two points where dynamical coherence is used in the proof. One is to be able to apply Lemma \ref{l.graphtransf}. The most crucial one is to establish Proposition \ref{prop-circleexist}, in particular, to show that if two $c$-curves start close, then they remain close in a compact set. This last point fails in general as can be seen in the example of \cite{HHU-example}. However, there is a statement that can be shown with essentially the same ideas (and we will use it for studying the absolute partially hyperbolic case) and works even if the $cs$ and $cu$-directions integrate into general branching foliations.

\begin{prop}\label{p.novertifcircle} 
Any minimal sublamination $\Lambda$ of the vertical $cs$-lamination $\Lambda^{cs}$ is disjoint from the vertical $cu$-lamination $\Lambda^{cu}$.  
\end{prop}

\section{Proof of Theorem \ref{teo-novert}}\label{sec-novertproof}

In this section we complete the proof of Theorem \ref{teo-novert}. This section is independent from section \ref{sec-coherent} except from subsection \ref{ss.absolute}. 

\subsection{Transitive case}\label{ss.trans}
In this subsection we establish Theorem \ref{teo-novert} under assumption (\ref{it-transitivity}). 

\begin{remark} 
As mentioned earlier, if $f$ is transitive (i.e. has a dense orbit) or is volume preserving, then it satisfies property (\ref{it-transitivity}). 
\end{remark}

\begin{prop}\label{p-transcase} If $f$ verifies property (\ref{it-transitivity}) of Theorem \ref{teo-novert} then $\Lambda^{cs}=\emptyset$. 
\end{prop}

\begin{proof}
Proposition \ref{prop-Lambdacompactinv} implies that $\Lambda^{cs}$ is a compact $f$-invariant set which is not entirely $M$. Choose some small $\eps>0$ and consider the set

$$ U_\eps= \bigcup_{x \in \Lambda^{cs}} \Fu_{\eps}(x)  $$ 
Here, $\Fu_\eps(x)$ denotes the open $\eps$-neighborhood of $x$ in its strong unstable leaf with the induced metric. 

As $f^{-1}(\Fu_\eps(x)) \subset \Fu_{\lambda \eps}(f^{-1}(x))$ for some $\lambda<1$, it is enough to show that $U_\eps$ is open to obtain that $f^{-1}(\overline U_\eps) \subset U_\eps$. Since $f$ is chain-recurrent, this implies that if $\Lambda^{cs}\neq \emptyset$ then  $U_\eps =M$ for all $\eps>0$, but this is impossible since $\Lambda^{cs}\neq M$.  

To show that $U_\eps$ is open, recall that $\Lambda^{cs}$ is saturated by leaves of $\Fcs$ and therefore at each $x \in \Lambda^{cs}$ one can choose a small open disk $D$ inside a leaf of $\Fcs$ contained in $\Lambda^{cs}$ and containing $x$ and by local product structure one has that $\bigcup_{y \in D} \Fu_\eps(y)$ is an open set. 
\end{proof}

\subsection{Homotopic to identity case}\label{ss.homotId}

Here we establish item (\ref{it-isotopicid}) of Theorem \ref{teo-novert}. We consider $f:M\to M$ to be homotopic to the identity. Then, we can lift $f$ to the universal cover $\tilde M$ of $M$ to obtain a diffeomorphism $\tilde f: \tilde M \to \tilde M$ which is at bounded distance from the identity. 

Modulo taking an iterate, we can consider a fixed vertical $cs$-leaf $L$ (given by Lemma \ref{l.verticallamination}(\ref{it-periodic})) whose lift $\tilde L$ is quasi-isometrically embedded thanks to Lemma \ref{l.verticallamination} (\ref{it-quasiisomleaves}). It follows that for each strong stable arc $J \subset \tilde L$, the  diameter of $\tilde f^{-n}(J)$ remains bounded by a polynomial in $n$ (in fact, the growth is at most linear). On the other hand, as the area of a ball of radius $R$ in $\tilde L$ has area bounded by a polynomial in $R$ (cf. Lemma \ref{l.verticallamination} (\ref{it-polynomialarea})), one has that the area of the $\eps$-neighborhood of $\tilde f^{-n}(J)$ in $\tilde L$ is bounded by a polynomial in $n$.

This gives a contradiction with Proposition \ref{p-volvslength} as the length of $\tilde f^{-n}(J)$ is exponentially big with $n$.

\begin{remark}
This proof also holds if the induced action of $f$ in the base is reducible with no pseudo-Anosov component (see \cite[Chapter 1]{Calegari}). 
\end{remark}

\subsection{Homotopic to Pseudo-Anosov}\label{ss.homotPA}

We first show the following generalisation of the main result in \cite{BI}. 

\begin{prop}\label{p-volvslengthsmall}
Assume that $f: M \to M$ is a partially hyperbolic diffeomorphism of a 3-manifold $M$ and $U \subset M$ is an open invariant set such that the image of the morphism $\pi_1(U) \to \pi_1(M)$ induced by the inclusion of $U$ in $M$ is abelian. Assume moreover that there exists an unstable arc $I \subset U$ such that the iterates of $I$ remain at distance $\geq \eps$ from $\partial U$. Then, the action of $f_\ast$ on $\pi_1(U)$ has an eigenvalue larger than $1$. 
\end{prop}

\begin{proof}
Fix $\tilde U$, a connected component of the lift of $U$ to the universal cover $\tilde M$ of $M$ which is fixed by $\tilde f$ a lift of $f$. If $f_\ast$ does not have eigenvalues of modulus larger than one, it follows that the diameter of a compact set in $\tilde U$ when iterated forward by $\tilde f$ grows at most polynomially (see \cite[Section 2]{BI}). Moreover, the intersection of the ball of radius $R$ in $\tilde M$ with $\tilde U$ has volume which is polynomial in $R$ since the set of deck transformations which fix $\tilde U$ is an abelian subgroup of $\pi_1(M)$. 

Now, if one considers an unstable arc $J \subset \tilde U$  which projects to $I$ as in the statement of the proposition, its iterate $\tilde f^n(J)$ has length which is exponentially big in $n$. As the diameter of $\tilde f^n(J)$ is bounded by a polynomial in $n$ and its $\eps$-neighborhood is contained in $\tilde U$, one reaches a contradiction with Proposition \ref{p-volvslength}. 
\end{proof}

Notice now that if the vertical lamination is non-empty and $f$ acts as pseudo-Anosov in the base surface, then, the vertical lamination has to be \emph{full} (see \cite[Chapter 1]{Calegari} for a proof and more information on these concepts). In particular, the fundamental group of the complement is cyclic (and generated by the center of $\pi_1(M)$). Using Lemma \ref{l.verticallamination} (\ref{it-unst}) one is in the hypothesis of Proposition \ref{p-volvslengthsmall} which gives a contradiction as $f_\ast$ acts as the identity in the center of $\pi_1(M)$. This completes the proof of item (\ref{it-isotopicpA}) of Theorem \ref{teo-novert}.

\subsection{Absolute case}\label{ss.absolute}
In this section we establish item (\ref{it-absolute}) of Theorem \ref{teo-novert}. 

As before, we will assume that  $\Lambda^{cs}$ is non-empty and work in a periodic leaf $L$ (cf. Lemma \ref{l.verticallamination} item (\ref{it-periodic})) which we assume for simplicity that is fixed. We will say that a curve in $L$ is a $c$-curve if it is contained in the intersection of $L$ with a $cu$-leaf of the corresponding branching foliation. 

Using Proposition \ref{p.novertifcircle} one can see that inside $L$ any $c$-curve intersects any $s$-curve in at most one point (otherwise, one can proceed as in Proposition \ref{prop-circleexist} to obtain a circle $c$-leaf). By local product structure, this implies:

\begin{prop}\label{prop-vollengthcs}
There is $K_0>0$ and $\eps>0$ such that if $J$ is any $s$ or $c$-curve, then:

$$ \mathrm{length}(J) < K_0 \, \mathrm{area}(B_\eps(J)). $$
\end{prop}

As $L$ is vertical we can isotope the fibers so that $L$ is a union of fibers (Theorem \ref{teo-brit}) and without loss of generality assume each fiber has length bounded by $1$. For $x,y \in L$, let $C_x$ denote the fiber through $x$ and $d(x,y)$ the distance measured inside $L$. The above proposition gives:

\begin{cor}\label{cor-vollen}
There is $K>0$ with the following property. If $x,y \in L$ with $d(x,y)>2$, $A_{x,y}$ is the closed annulus between $C_x$ and $C_y$, and $J$ is a $c$ or $s$-curve in $A_{x,y}$ joining $C_x$ to $C_y$, then $\mathrm{length}(J) \leq K d(x,y)$. 
\end{cor}

Absolute partial hyperbolicity implies the existence of $\sigma<\mu$ such that 

$$ \|Df|_{\Es}\| < \sigma< \|Df|_{\Ec}\| <\mu, \  \ \forall x \in L. $$

Moreover, from compactness, one has that $\|Df\| > \lambda>0$ for all $x\in L$. 

As there are no circle $s$-leaves and Proposition \ref{p.novertifcircle} implies there are no circle $c$-leaves one gets by a Poincar\'e-Bendixson argument that given two fibers there are $s$ and $c$ curves between them. 

Let $D>0$ be a very big constant (to be chosen at the end) and take points $x,y \in L$ with $d(x,y)=D$. Let $J^s$ be a stable curve in $A_{x,y}$ joining $C_x$ and $C_y$ with endpoints $x^s,y^s$, One deduces from Corollary \ref{cor-vollen} that $\mathrm{length}(J^s)  \leq KD$ and that $\mathrm{length}(f^n(J^s)) \leq K\sigma^n D$ for some $n\geq 1$ (to be specified at the end). 

Consider two fibers through the endpoints of $f^n(J^s)$ and choose a $c$-curve $J^c$ joining those fibers with endpoints $x^c,y^c$ inside the region delimited by them. Using Corollary \ref{cor-vollen} again one gets 
\[
 \mathrm{length}(J^c) \leq K \, \mathrm{length}(f^n(J^s)) \leq K^2 \sigma^n D. 
\]
Then, it follows that:
\[
 \mathrm{length}(f^{-n}(J^c)) \leq \frac{K^2 \sigma^n}{\mu^n} D. 
\]
As $f^n(x^s)$ lies in the same fiber as $x^c$ one gets:
\[
 d(x^s , f^{-n}(x^c)) \leq \lambda^{-n}, 
\]
\noindent and an identical estimate for $d(y^s,f^{-n}(y^c))$. The triangle inequality provides: 

$$ D-2 \leq d(x^s,y^s) \leq \frac{K^2 \sigma^n}{\mu^n} D + 2 \lambda^{-n}. $$

Choose $n>0$ so that $\frac{K^2 \sigma^n}{\mu^n} < \frac{1}{2}$.
Then choose $D$ large enough to produce a contradiction.

\section{Circle bundles: Proof of Theorems A and B}\label{s.proofCircle}

Admitting a horizontal foliation already imposes topological restrictions on a circle bundle in the form of the Milnor-Wood inequality \cite[Book II, Chapter 4]{CandelConlon}.
In this section we exploit that the $cs$-foliation is not only horizontal, but
also admits a non-zero vector field tangent to it.
This imposes even stronger obstructions on the topology of the circle bundle,
and implies that it finitely covers the unit tangent bundle of a surface.

\begin{remark}\label{rem-jonathanbowden}
Notice that there exist foliations on circle bundles for which some leaves are vertical but for which their tangent distributions are \emph{homotopic} to horizontal ones. As our arguments are from differential topology, the obstructions carry on to this setting, so it would be enough to show that for a general partially hyperbolic diffeomorphism, the center-stable (resp. center unstable) bundle is homotopic to a horizontal one to obtain the obstructions here. %
\end{remark}

In this section we state our main result for the case of an orientable circle bundle over an orientable surface (see subsection \ref{ss.circlebundles}). Throughout this section, $\Sigma$ will denote a closed and orientable surface of genus greater than one. We say that the bundle $M\rightarrow\Sigma$ is orientable if $M$ is orientable as a 3-manifold. We first give a proof in the case $M=\Sigma \times S^1$ and then continue with the proof of the general case. The proof is organized as follows: in subsection \ref{ss.hurwitz} we show an abstract result giving an obstruction for admitting horizontal vector fields in a circle bundle. Then in subsection \ref{ss.ABorientable}, we use the obstruction to show Theorems A and B assuming that the bundles of the partially hyperbolic splitting are orientable. Finally in subsection \ref{ss.nonorientable}, we remove this hypothesis.%

\subsection{The trivial bundle case}

Here we deal with the following case which admits a simple proof. Even if it also follows from the more general setting we felt it deserved to be explained first, as the general idea is in some sense a generalisation of this case. 

\begin{teo}\label{teo-product}
For any surface $\Sigma$ of genus $g\geq 2$ the manifold $M=\Sigma \times S^1$ does not admit a partially hyperbolic diffeomorphism without vertical leaves (cf. Theorem \ref{teo-novert}). 
\end{teo}

\begin{figure}[ht]
\vspace{-0.5cm}
\begin{center}
\includegraphics[scale=0.4]{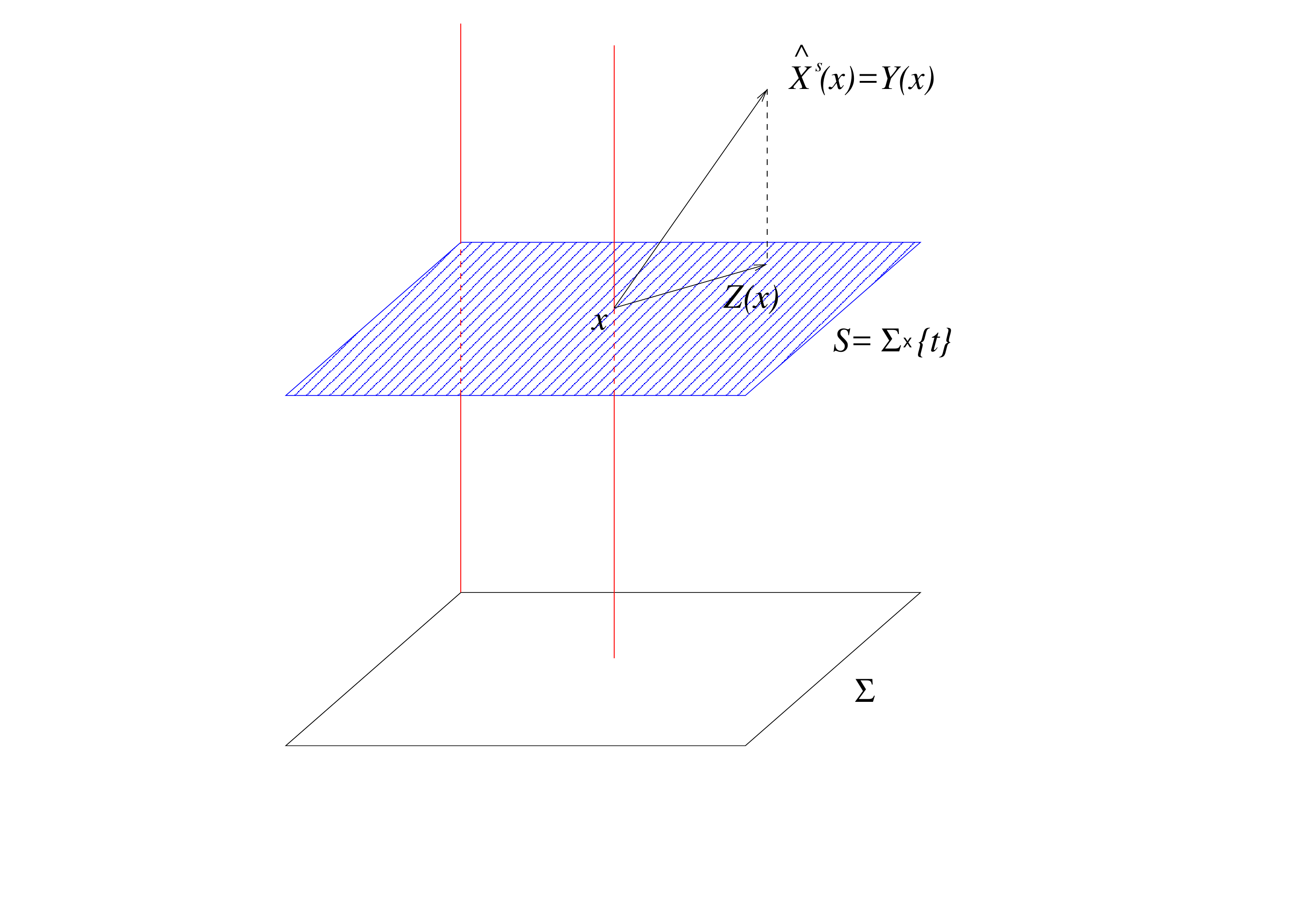}
\begin{picture}(0,0)
\end{picture}
\end{center}
\vspace{-0.5cm}
\caption{Constructing a vector field in a section of the bundle.\label{f.trivialbundle}}
\end{figure}

\begin{proof} We assume by contradiction that there exists such a partially hyperbolic diffeomorphism.  We can assume without loss of generality that all bundles are orientable and the orientation is preserved by $f$ (otherwise, one takes a finite cover and an iterate). As in Corollary \ref{cor-novert}, we know that there is a projection $p_\eps$ isotopic to $p$ so that every leaf of the approximating foliation $\cW_\eps$ is transverse to the fibers. 

Consider a non-singular vector field $X^s$ generating the $E^s$-bundle.
On each leaf of the approximating foliation $\cW_\eps$,
the map $h_\eps$ allows one to pull back the vector field
$X^s$ to a vector field $\hat X^s$ (cf. Fact \ref{fact-diffeo}).
The result is a continuous vector field $\hat X^s$ tangent to $\cW_\eps$.

Choose a surface $S=\Sigma\times \{t\}$ in $M$ embedded orthogonally to all the fibers of $p_\eps$ (i.e. $S$ is a section of the bundle $p_\eps:M\to\Sigma$), and let's call $Y$ to the restriction of the vector field $\hat{X}^s$ to $S$. Being horizontal, the projection of $Y$ onto $TS$ provides a non-vanishing vector field $Z$, contradicting that $S$ is a higher genus surface. This completes the proof. 

\end{proof}

\subsection{Classifying circle bundles via the Euler number}\label{ss.hurwitz}

Recall from subsection \ref{ss.circlebundles} that the fundamental group of a circle bundle $M$ over a genus $g$ ($g \geq 2$) surface $\Sigma$ admits the following presentation, where $\mathrm{eu}(M)$ is the \emph{Euler number} of the bundle:

\[ \pi_1(M) = \bigg\langle a_1,b_1, \ldots, a_g, b_g, c \  | \  \prod_{i=1}^g [a_i,b_i] = c^{\mathrm{eu}(M)} \ , \ [a_i,c]= \mathrm{id} \ , \ [b_i,c]=\mathrm{id} \bigg \rangle.  \]  %

It turns out that the topology of the 3-manifold $M$ is determined by the integer $\mathrm{eu}(M)$ 
according to the following proposition (see \cite{Hatcher}). 

\begin{prop}
Let $\Sigma$ be a closed and orientable surface (of any genus). Then
\begin{enumerate}
\item For every $n\in\mathbb{Z}$ there exists a circle bundle $M\rightarrow\Sigma$ such that $\mathrm{eu}(M)=n$;
\item Two circle bundles over $\Sigma$ are homeomorphic under an orientation preserving homeomorphism if and only if both have equal Euler number.
\item If $M$ is an oriented circle bundle over $\Sigma$ and $-M$ denotes the same manifold but with the opposite orientation, then $\mathrm{eu}(-M)=-\mathrm{eu}(M)$.  
\end{enumerate} 
\end{prop}

As a corollary we get that the homeomorphism class of circle bundles over $\Sigma$ can be classified by the positive integer $|\mathrm{eu}(M)|$. It holds that the Euler number of the unitary tangent bundle of $\Sigma$ equals the Euler characteristic of $\Sigma$, i.e.,
$|\mathrm{eu}(T^1\Sigma)|=|\chi(\Sigma)|$, and also $\mathrm{eu}(M\times S^1)=0$. 

The next proposition can be interpreted as a generalization of the Hurwitz formula for covering maps between surfaces which states that if $h:\hat \Sigma \rightarrow \Sigma$ is a covering map, then $\chi(\hat \Sigma)=|\mathrm{deg}(h)|\cdot\chi(\Sigma)$. 

Let $\hat M \rightarrow \hat \Sigma$ and $M \rightarrow \Sigma$ be two circle bundles and $H:\hat M\rightarrow M$ a fiber preserving map. Observe that by taking the quotient of $\hat M$ and $M$ by its fibers, this map gives rise to another map $h: \hat \Sigma\rightarrow\Sigma$ such that the following diagram commutes:
\begin{center}
$\begin{CD}
\hat M     @>H>>  M\\
@VVV        @VVV\\
\hat \Sigma     @>h>>  \Sigma
\end{CD}$
\end{center}
Observe also that if a fiber $\hat S$ of $\hat M$ is sent into a fiber $S$ of $M$ by $H$, then the degree of $H$ restricted to $\hat S$ is constant along all the fibers of $\hat M$. We will restrict to the case where $h=\mathrm{id}$, though the formula can be easily generalised to more general settings. Notice that in this case, the degree of $H$ coincides with the degree in the fibers. 

\begin{prop}\label{prop-Hurwitz}
In the previous setting with $h=\mathrm{id}$, it holds that 
\[\mathrm{deg}(H) \mathrm{eu}( \hat M)=\mathrm{eu}(M).\]
\end{prop} 

This is a special case of a more general theory of Euler numbers
and covering maps.
See, for instance, \cite[\S 3]{jn1983lectures}. 
For completeness, we include a short proof using the fundamental groups.

\begin{proof}
Consider the following presentations for the fundamental groups of $\hat M$ and $M$:

\[ \pi_1(\hat M) = \bigg \langle \hat a_1,\hat b_1, \ldots, \hat a_g, \hat b_g, \hat c \  | \  \prod_{i=1}^g [\hat a_i,\hat b_i] = \hat c^{eu(\hat M)} \ , \ [\hat a_i,\hat c]= \mathrm{id} \ , \ [\hat b_i,\hat c]=\mathrm{id} \bigg \rangle ,  \]
 \[ \pi_1(M) = \bigg \langle a_1,b_1, \ldots, a_g, b_g, c \  | \  \prod_{i=1}^g [a_i,b_i] = c^{eu(M)} \ , \ [a_i,c]= \mathrm{id} \ , \ [b_i,c]=\mathrm{id} \bigg \rangle.  \]

As $H$ lifts the identity, the action on the presentation of the fundamental group we wrote above verifies:
\begin{itemize}
\item for all $i=1,\ldots, g$ there exists $i_a$ so that  $H_\ast (\hat a_i) =a_i c^{i_a}$,
\item for all $i=1,\ldots, g$ there exists $i_b$ so that  $H_\ast(\hat b_i)=b_i c^{i_b}$,
\item and $H_\ast(\hat c)=c^n$ where $n=\mathrm{deg}(H)$. 
\end{itemize}

Notice that this implies that $H_\ast([\hat a_i,\hat b_i]) = [a_i,b_i]$ for all $i=1,\ldots, g$. 

It follows that if $\hat m = \mathrm{eu}(\hat M)$ and $m=\mathrm{eu}(M)$ then:

\[ c^{ m} =  \prod_{i=1}^g [ a_i, b_i]  =  \prod_{i=1}^g [H_\ast(\hat a_i), H_\ast(\hat b_i)] = H_\ast \left(  \prod_{i=1}^g [\hat a_i,\hat b_i] \right) = H_\ast( \hat c^{ \hat m}) = \hat c^{n \hat m}, \]

 \noindent which implies $n \hat m = m$ as desired. 
\end{proof}

\subsection{Proof of Theorems A and B: orientable bundles}\label{ss.ABorientable}

Theorems A and B are a consequence of the following more general result: 

\begin{teo}\label{teo-horizimpliesanosov} 
Assume that $M$ is an orientable circle bundle over an orientable surface $\Sigma$ of genus $g\geq 2$. Assume that $M$ admits a partially hyperbolic diffeomorphism such that the $cs$-foliation has no vertical leaves (in the sense of Theorem \ref{teo-novert}). Then,
$|\mathrm{eu}(M)| = \frac{2g-2}{n}$ for some $n$ which divides $2g-2=-\chi(\Sigma)$. 
\end{teo}

In particular, this theorem applies to partially hyperbolic diffeomorphisms in the hypothesis of Theorem \ref{teo-novert} and therefore it implies both Theorem A and B. 

In this section we prove Theorem \ref{teo-horizimpliesanosov} under the additional assumption that the bundles $\Es,\Ec,\Eu$ are orientable so that we can apply Theorems \ref{teoBI1} and \ref{teoBI2}. In the next subsection we show that this is always the case. 

\medskip

\begin{figure}[ht]
\vspace{-0.5cm}
\begin{center}
\includegraphics[scale=0.35]{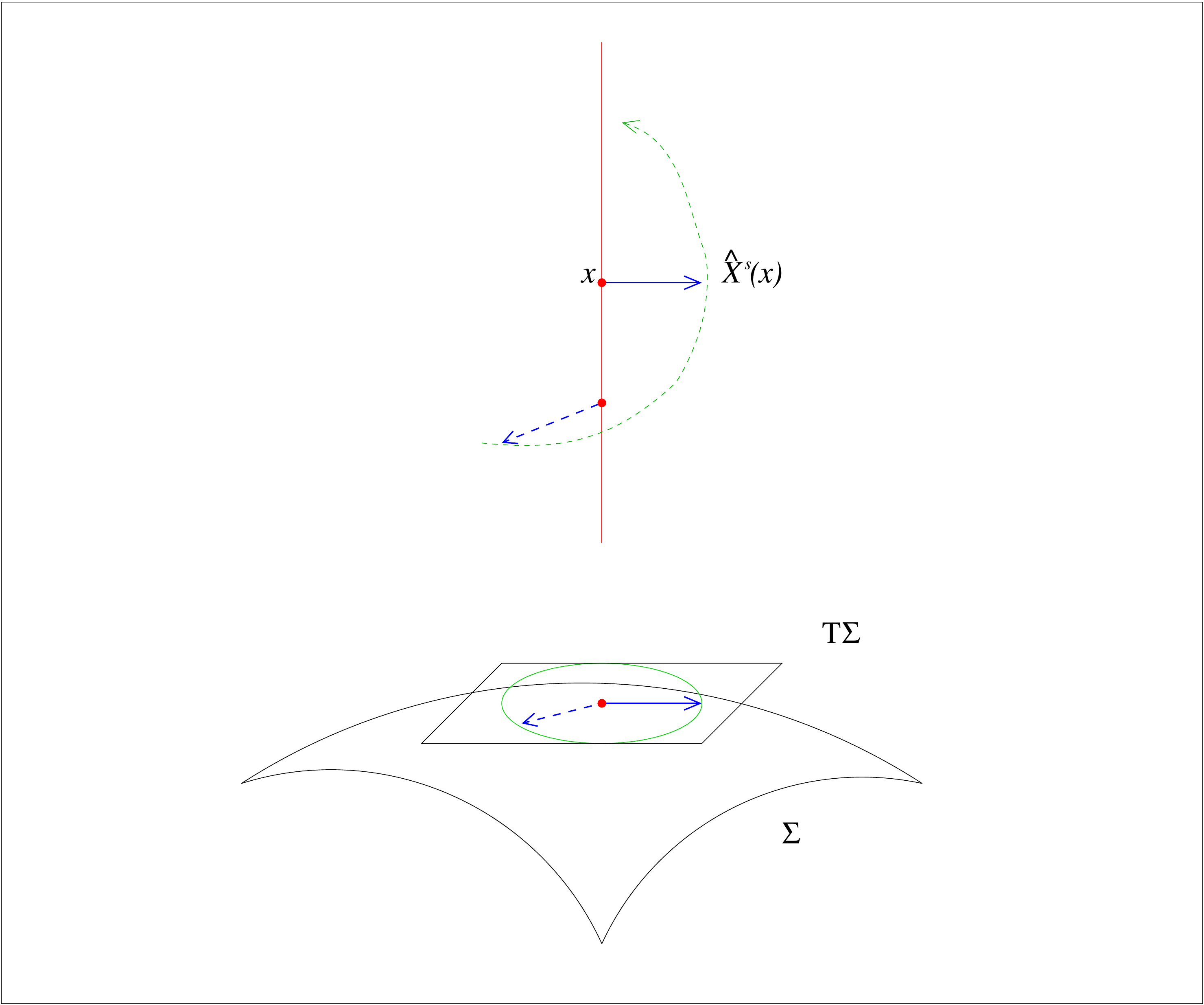}
\begin{picture}(0,0)
\end{picture}
\end{center}
\vspace{-0.5cm}
\caption{A map from the bundle to the unit tangent bundle.\label{f.generalbundle}}
\end{figure}

\medskip

\begin{proof}  As in Corollary \ref{cor-novert} we know that there is a projection $p_\eps$ isotopic to $p$ so that every leaf of the approximating foliation $\cW_\eps$ is transverse to the fibers. 

Consider a non-singular vector field $X^s$ generating the $\Es$-bundle. The map $h_\eps$ allows one to pull back this vector field  (cf. Fact \ref{fact-diffeo}) to obtain a continuous vector field $\hat X^s$ in each leaf of the approximating foliation $\cW_\eps$. Moreover, as $h_\eps$ is a diffeomorphism when restricted to each leaf of $\cW_\eps$ we have that the vector field $\hat X^s$ is uniquely integrable (because $\Es$ is). Let $\phi^t$ be the flow generated by $\hat X^s$. 

We choose a fixed Riemannian metric in $\Sigma$ and fix $\delta>0$ small so that, for every $x \in M$ one has that $p_\eps(\phi^\delta(x))$ is in a small neighbourhood of $p_\eps(x).$ As $\cW_\eps$ is horizontal, we know then that $p_\eps(\phi^\delta(x))\neq x$ and so we can construct for each $x \in M$ a vector $v \in T_{p_\eps(x)}\Sigma$ tangent to the geodesic from $p_\eps(x)$ to $p_\eps(\phi^\delta(x))$. 

We obtain a map $H: M \to T^1\Sigma$ which is continuous as all the choices were continuous. Moreover, one has that if $p: T^1 \Sigma \to \Sigma$ is the canonical projection then $p \circ H = p_\eps$. Applying Proposition \ref{prop-Hurwitz} we complete the proof. 
\end{proof}

\begin{remark}
If the projection $p_\eps$ could be chosen to be smooth, then one could just take the map $M \to T^1\Sigma$ given by $x \mapsto D_xp_\eps(\hat X^s(x))$. As it is not a priori the case, we need to make some slight adjustments to make the idea work. See figure \ref{f.generalbundle}. Notice that Theorem \ref{teo-horizimpliesanosov} does not really require the foliations to come from a partially hyperbolic diffeomorphism, just to have a horizontal foliation admitting a continuous vector field tangent to it. 
\end{remark}

\subsection{The bundles are orientable}\label{ss.nonorientable}

We consider the more general case where the bundles $\Es,\Ec,\Eu$ can be non-orientable. We let $q: \hat M \to M$ be the smallest degree covering which makes all the bundles orientable and choose $\hat f: \hat M \to \hat M$ a lift of some iterate of $f$. In the covering, Theorems \ref{teoBI1} and \ref{teoBI2} apply and so we may use the results of Theorem \ref{teo-novert} to show that the branching foliations are horizontal (notice that hypothesis (1)-(5) in Theorem \ref{teo-novert} still hold when one lifts to a finite cover). 

As $M$ is orientable and $\Sigma$ too, this means that $\hat M$ is an orientable circle bundle over an orientable surface $\hat \Sigma$. This implies that $\hat \Ecs$ and $\hat \Ecu$, being horizontal, inherit the orientation of $\hat \Sigma$ and as deck transformations preserve the orientation of $\hat M$ and of $\hat \Sigma$ they must preserve the orientation of the bundles $\hat \Ecs, \hat \Ecu, \hat \Eu, \hat \Es$ which implies that the bundles $\Es,\Ec,\Eu$ were orientable in the first place. This concludes the proof. 
This same argument involving orientation also appears in Lemma
\ref{lemma:orient} in the next section.

\section{Generalization to Seifert fiber spaces}\label{s.Seifert}

This section proves Theorem C %
and other results specific to Seifert
fiberings.

Consider a partially hyperbolic diffeomorphism $f_2$ defined
on a closed 3-mani\-fold $M_2$ with a Seifert fibering $M_2  \to  \Sig_2$
where $\Sig_2$ is a hyperbolic orbifold.
First consider the case where everything is oriented.
That is, $M_2$, $\Sig_2$, the fibering, and the bundles $\Eu$, $\Ec$, and $\Es$ all have
orientations.
Up to replacing $f_2$ by an iterate, further assume that
$f_2$ preserves the orientation of the bundles.
By Theorem 2.2,
there is a branching foliation $\Fcs_2$ tangent to $\Ecs$.
We consider a leaf $L$ of $\Fcs_2$ to be \emph{vertical} if the embedding of its
fundamental group $\pi_1(L)$ into $\pi_1(M_2)$ contains the cyclic subgroup
generated by a regular fiber of the Seifert fibering.
Our first goal is to show that $\Fcs_2$ has no vertical leaves.

As $\Sig_2$ is a hyperbolic orbifold,
there is a finite (orbifold) covering $\Sig_1  \to  \Sig_2$
where $\Sig_1$ is a closed hyperbolic surface
\cite[Theorem 5.1.5]{cho2012geometric}.
This induces a fiber-preserving covering
$M_1  \to  M_2$ where $M_1$ is a circle bundle.
An iterate of $f_2$ lifts to a partially
hyperbolic map $f_1$ on $M_1$
and
the branching foliation $\Fcs_2$ lifts to a branching foliation $\Fcs_1$
on $M_1$.

Assume $f_2$ satisfies one of the properties (1)--(5) listed in Theorem 3.1.
Then $f_1$ also satisfies the corresponding property
and Theorem 3.1 implies that $\Fcs_1$ has no vertical leaves.
As the covering $M_1  \to  M_2$ preserves fibers,
$\Fcs_2$ has no vertical leaves.
The foliation results of Brittenham and Thurston (Theorem \ref{teo-brit})
hold in the setting of Seifert fiberings, and
we may apply an isotopy to the fibering
so that the fibers are transverse to the branching foliation.
The techniques in the proof of Theorem \ref{teo-horizimpliesanosov}
are entirely local in nature
and may be used here to construct a vector field
which is nowhere tangent to the fibers.
The results in \cite{ham20XXhorizontal}
(which generalize results in the current paper to the Seifert fibered setting)
then show that $M_2$ is a covering space of $\UTSig_2$.
This concludes the proof of Theorem C in the case where everything is
orientable.
Before going to the non-orientable setting, we first analyse
fiber-pre\-ser\-ving
diffeomorphisms on $M_2$.

\begin{lemma} \label{lemma:orient}
    Suppose $\tau$ : $M_2  \to  M_2$ is a fiber-preserving diffeomorphism
    which preserves the partially hyperbolic splitting for which $\Ecs$ and $\Ecu$ are horizontal.
    Then the following are equivalent:{}
    \begin{enumerate}
        \item $\tau$ preserves the orientation of any one (and therefore all)
        of\\
        $\Eu$, $\Es$, $\Ecu$, $\Ecs$;

        \item $\tau$ preserves the orientation of the Seifert fibering;

        \item the induced map $\Sig_2  \to  \Sig_2$ preserves orientation.
          \end{enumerate}  \end{lemma}
\begin{proof}
    If $\tau$ reverses the orientation of $M_2$,
    then it is a map of degree -1.
    This implies that the Euler number of $M_2$ equals zero.
    Then the covering space $M_1$ defined above
    also has Euler number zero
    which we have already ruled out in the case of circle bundles.
    Thus, $\tau$ must preserve the orientation of $M_2$.
    (See also \cite[Theorem 3.1]{ehn1981transverse}.)

    Since $\Ecs$ is transverse to the fibering
    and both $\Ecs$ and the fibering are oriented,
    the derivative $D \tau$ must either preserve both of these orientation
    or reverse both.
    As $\Ecs$ is transverse to $\Eu$,
    $D \tau$ must either preserve both orientations or reverse both.
    Similar arguments
    hold for $\Ecu$ with the fibering and for $\Ecu$ with $\Es$.
    Altogether, these show that (1) and (2) are equivalent.

    Since $\tau$ preserves the orientation of $M_2$, (2) and (3) are equivalent as the distributions $\Ecs$ and $\Ecu$ (being horizontal) carry an orientation which is in correspondence with the orientation of $\Sigma$.
\end{proof}
Now, suppose $f_3$ is a partially hyperbolic map defined on a
manifold with Seifert fibering $M_3  \to  \Sig_3$
where $\Sig_3$ is an oriented hyperbolic orbifold.
Then there is a finite normal covering $M_2  \to  M_3$
and a lift $f_2$ of an iterate of $f_3$ so that the fibering on $M_2$
and all of the bundles
in the splitting $TM_2 = \Eu \oplus \Ec \oplus \Es$ are oriented.
The fibering on $M_2$ is given by a map of the form $M_2  \to  \Sig_2$
where $\Sig_2$ orbifold covers $\Sig_3$.
Since $\Sig_3$ is oriented,
any deck transformation $\tau$ : $M_2  \to  M_2$ of the covering 
$M_2  \to  M_3$ induces an orientation preserving map on $\Sig_2$.
\Cref{lemma:orient} then shows that $\tau$ preserves the orientations
of $\Eu$, $\Ec$, and $\Es$, and the fibering.
Since this holds for all deck transformations, it implies that
all of the bundles in the splitting
$TM_3 = \Eu \oplus \Ec \oplus \Es$ are oriented as well as the fibering on $M_3$.
Thus, $f_3$ on $M_3$ is already in the ``everything is oriented'' case
and a lift to $M_2$ is not necessary.
This shows that Theorem C holds for any Seifert fiber space over an
oriented hyperbolic orbifold.

\medskip

The final case is a partially hyperbolic map $f_4$
defined on a Seifert fibering $M_4  \to  \Sig_4$ where
$\Sig_4$ is a non-orientable hyperbolic orbifold.
Let $\Sig_3$ be the oriented double cover of $\Sig_4$.
Then the map $\Sig_3  \to  \Sig_4$ defines a fiber-preserving covering
$M_3  \to  M_4$ and a lifted map $f_3$ on $M_3$.
(See \cite[Corollary 3.2]{jn1983lectures} for details.)
Since $M_3$ is a double cover of $M_4$,
there is a deck transformation $\tau : M_3 \to M_3$
which induces an orientation reversing map on $\Sig_3$.
\Cref{lemma:orient} implies that $\tau$ reverses the orientation of fibers.
As it is a deck transformation, $\tau$ has no fixed points and therefore
does not fix any fibers.
This implies that $M_3$ has exactly two copies of every exceptional fiber
in $M_4$ and
the Euler numbers satisfy
\begin{math}
    eu(M_3 \to \Sig_3) = 2 eu(M_4 \to \Sig_4).
\end{math}
Using these properties and the fact that $M_3$ is a $d$-fold cover of $\UTSig_3$,
one may show that $M_4$ has the same Seifert invariant as
a $d$-fold cover of $\UTSig_4$.
Since the Seifert invariant uniquely determines the space,
$M_4$ finitely covers $\UTSig_4$.
This concludes the proof of Theorem C for a Seifert fibering over any
hyperbolic orbifold.

\bigskip

\begin{proof}[Proof of \cref{thm:turnover}]
    Suppose $M  \to  \Sig$ is a Seifert fibering over a turnover.
    If $\Sig$ is not hyperbolic, the results in
    \cite[Appendix A]{HPSol}
    imply that $M$ does not support a partially hyperbolic system.
    Therefore, we restrict to the case that $\Sig$ is hyperbolic.
    Then, \cite[Corollary 4 and Proposition 6]{Brittenham}
    imply that the approximating foliation to the $cs$-branching foliation
    has no vertical leaves.
    As in the last proof,
    one can show that $M$ finitely covers $T^1 \Sig$
    and therefore $M$ supports an Anosov flow.
\end{proof}
\begin{proof}[Proof of \cref{cor:turnex}]
    In this proof, we write Seifert invariants using the same conventions as in
    \cite{jn1983lectures}.
    Consider a Seifert fiber space over a hyperbolic turnover
    such that the fibering $M  \to  \Sig$ has
    Seifert invariant
    \[
        (0; (1,-1), (\al_1, \bt_1), (\al_2, \bt_2), (\al_3, \bt_3))
    \]
    with $0 < \bt_i < \al_i$ for all $i$.
    Up to a choice of orientation,
    the Seifert invariant of the unit tangent bundle $T^1 \Sig$
    may be written as
    \[
        (0; (1,-1), (\al_1, 1), (\al_2, 1), (\al_3, 1)).
    \]
    Suppose $M$ covers $T^1 \Sig$.
    As both manifolds have unique Seifert fiberings
    \cite[Theorem 5.2]{jn1983lectures},
    we may assume this covering is fiber preserving.
    As both $M$ and $T^1 \Sig$ have the same base orbifold,
    the formula for coverings given in
    \cite[Proposition 2.5]{jn1983lectures} implies that
    the fibering on $T^1 \Sig$ has Seifert invariant
    \[
        (0; (1,-d), (\al_1, d \bt_1), (\al_2, d \bt_2), (\al_3, d \bt_3))
    \]
    where the integer $d$ is the degree of the covering.
    The properties of Seifert invariants then imply that
    $d \bt_i \equiv 1 \mod \al_i$ for all $i$.
    For most choices of $\bt_i$, no such $d$ with this property exists
    and so there are many Seifert fiber spaces over $\Sig$
    which do not cover $T^1 \Sig$.

    The exact conditions for $M$ to support a horizontal foliation are
    complicated to state \cite{nai1994foliations}.
    See also \cite[Section 4.11]{Calegari}. %
    However, a sufficient condition is that the inequality
    \[ 
         \frac{\bt_1}{\al_1} +
         \frac{\bt_2}{\al_2} +
         \frac{\bt_3}{\al_3} < 1
           \]
    holds \cite[Equation 5.2 and Proposition 5.3]{ehn1981transverse}.
    Using this, one may find many possibilities for $M$
    which support horizontal foliations,
    but do not cover $T^1 \Sig$.
    For instance, if $\al  \ge  5$ is an odd number, then
    \[
        (0; (1,-1), (\al, 1), (\al, 2), (\al, 1))
    \]
    gives such an example.
\end{proof}

\appendix

\section{Proof of Proposition \ref{prop-quasigeodesic}} 

In this appendix we prove Proposition \ref{prop-quasigeodesic}. The key point is to establish that every curve in the family is a quasi-geodesic which is the purpose of the following:

\begin{lemma}\label{lem-rayinfinity}
Let $\eta: \mathbb{R} \to \mathbb{H}^2$ be a properly embedded curve and $\Gamma < \mathrm{PSL}(2,\mathbb{R})$ the set of deck transformations for a given compact surface $\Sigma$. Assume that for every $g \in \Gamma$ it follows that $g \eta $ and $\eta$ do not have topologically transverse intersections. Then, $\lim_{t\to +\infty} \eta(t)$ exists in $\partial \mathbb{H}^2$. Moreover, if both components of $\mathbb{H}^2 \setminus \eta$ have infinite area, then $\lim_{t\to +\infty} \eta(t) \neq \lim_{t\to -\infty} \eta(t).$
\end{lemma}
\begin{proof} See Levitt \cite[Lemma 1]{Levitt} for a similar statement. %

Since the curve is properly embedded, the accumulation points of sequences $\eta(t_k)$ with $t_k \to \infty$ form a  non-empty compact connected subset of the boundary at infinity $\partial \mathbb{H}^2$ of $\mathbb{H}^2$. We will abuse notation and also use $\eta$ to denote the image $\eta(\mathbb{R})$.

Assume first by contradiction that these accumulation points form a non-trivial interval $[a,b]$ with $a\neq b$. Choose a deck transformation $g \in \Gamma$ such that $g [a,b] \subset (a,b)$. (Such a $g$ exists because $\mathbb{H}^2/_\Gamma$ is compact.) Notice that as $\eta$ and $g  \eta$ do not traverse each other, $g \eta$ cannot accumulate in $g [a,b]$ without accumulating in all the interval $(a,b)$. This gives a contradiction. 

To conclude, we must rule out that the set of accumulation points is the whole $\partial \mathbb{H}^2$. In this case, we choose $g \in \Gamma$. As $g$ is a hyperbolic element, it has an attractor and a repeller in $\partial \mathbb{H}^2$. One can choose a long arc $\eta_0 \subset \eta$ close to $\partial \mathbb{H}^2$ which can be closed to a Jordan curve $J$ by an arc $\alpha$ which is close to the boundary and far from the attractor and the repeller of $g$. If $\eta_0$ is sufficiently close to $\partial \mathbb{H}^2$ then $\alpha$ can be chosen in such a way that $g \alpha \cap \alpha = \emptyset$ and such that there is a point in the interior of $J$ which is mapped in the interior of $J$. This implies that $g$ or $g^{-1}$ sends the interior of the Jordan curve into itself which implies the existence of a fixed point of $g$ in $\mathbb{H}^2$ which is a contradiction. This completes the proof of the existence of $\lim_{t\to \infty} \eta(t)$. 

Let us call $\eta^+ = \lim_{t\to +\infty} \eta(t)$ and $\eta^- = \lim_{t \to -\infty} \eta(t)$. Assume that $z_0 = \eta^+ = \eta^-$ and let $D$ be the connected component of $\mathbb{H}^2 \setminus \eta$ whose only accumulation point in the boundary $\partial \mathbb{H}^2$ is $z_0$. 

Our hypothesis implies that $D$ has infinite area, and therefore there is a deck transformation $g \in \Gamma$ such that $g D \cap D \neq \emptyset$. One must have that $g z_0 = z_0$ as otherwise one would obtain a transverse intersection point between $g \eta$ and $\eta$. Without loss of generality, we assume that $z_0=g^-$ the repeller of $g$; this implies that say $g^{-1} D \subset D$ and if one considers $\hat D = \bigcup_{n>0} g^n(D)$ one has that $\hat D$ contains a properly embedded curve $\xi$ whose limits are $g^+$ and $g^-=z_0$ which are different. Choose $g' \in \Gamma$ such that $g' \xi \cap \xi$ transversally (i.e. the points at infinity are linked). This implies for large enough $n>0$ that $g' \circ g^n(D)$ intersects $g^n(D)$, and as $g' z_0 \neq z_0$ one obtains a transverse intersection of $\eta$ with $g^{-n} g' g^n \eta$, a contradiction. This completes the proof of the lemma.
\end{proof}

It follows from Lemma \ref{lem-rayinfinity} that every $\eta_i \in \mathcal{A}$ is $b_i$-quasi-isometric for some $b_i$. As the familly is closed by taking limits it is a compact family of curves and then the argument in \cite[Lemma 10.20]{Calegari} applies verbatim to give Proposition \ref{prop-quasigeodesic}.

\bigskip

\end{document}